\documentclass[11pt,english]{article}
\usepackage{lmodern}

\usepackage[T1]{fontenc}
\usepackage[latin9]{inputenc}
\usepackage{geometry}
\usepackage{float}
\usepackage{mathtools}
\usepackage{dsfont}
\usepackage{amsmath}
\usepackage{amsthm}
\usepackage{amssymb}
\usepackage{bm}
\usepackage{csquotes}

\makeatletter

\floatstyle{ruled}
\newfloat{algorithm}{tbp}{loa}
\providecommand{\algorithmname}{Algorithm}
\floatname{algorithm}{\protect\algorithmname}

\numberwithin{figure}{section}
\numberwithin{equation}{section}
\theoremstyle{plain}
\newtheorem{thm}{\protect\theoremname}
\theoremstyle{plain}
\newtheorem{lem}[thm]{\protect\lemmaname}
\theoremstyle{plain}
\newtheorem{assumption}[thm]{\protect\assumptionname}
\theoremstyle{definition}
\newtheorem{problem}[thm]{\protect\problemname}
\theoremstyle{definition}
\newtheorem{defn}[thm]{\protect\definitionname}
\theoremstyle{definition}
\newtheorem{rem}[thm]{Remark}
\theoremstyle{plain}

\theoremstyle{plain}
\newtheorem{cor}[thm]{\protect\corollaryname}

\usepackage{comment,url,algorithm,algorithmic,graphicx,subcaption,relsize}
\usepackage{amssymb,amsfonts,amsmath,amsthm,amscd,dsfont,mathrsfs,mathtools,microtype,nicefrac,pifont}
\usepackage{float,psfrag,epsfig,color,url,hyperref}
\hypersetup{
  colorlinks,
  linkcolor={red!50!black},
  citecolor={blue!50!black},
  urlcolor={blue!80!black}
}
\usepackage{upgreek}
\usepackage[dvipsnames]{xcolor}
\usepackage{epstopdf,bbm,mathtools,enumitem}
\usepackage[toc,page]{appendix}
\usepackage{dsfont}

\usepackage[mathscr]{euscript}
\allowdisplaybreaks
\usepackage{custom}

\makeatletter
\renewcommand{\paragraph}{%
  \@startsection{paragraph}{4}%
  {\z@}{1.25ex \@plus 1ex \@minus .2ex}{-1em}%
  {\normalfont\normalsize\bfseries}%
}
\makeatother

\makeatother

\usepackage{babel}
\providecommand{\assumptionname}{Assumption}
\providecommand{\definitionname}{Definition}
\providecommand{\lemmaname}{Lemma}
\providecommand{\problemname}{Problem}
\providecommand{\propositionname}{Proposition}
\providecommand{\theoremname}{Theorem}
\providecommand{\corollaryname}{Corollary}

\begin{document}
\def\balign#1\ealign{\begin{align}#1\end{align}}
\def\baligns#1\ealigns{\begin{align*}#1\end{align*}}
\def\balignat#1\ealign{\begin{alignat}#1\end{alignat}}
\def\balignats#1\ealigns{\begin{alignat*}#1\end{alignat*}}
\def\bitemize#1\eitemize{\begin{itemize}#1\end{itemize}}
\def\benumerate#1\eenumerate{\begin{enumerate}#1\end{enumerate}}

\newenvironment{talign*}
 {\let\displaystyle\textstyle\csname align*\endcsname}
 {\endalign}
\newenvironment{talign}
 {\let\displaystyle\textstyle\csname align\endcsname}
 {\endalign}

\def\balignst#1\ealignst{\begin{talign*}#1\end{talign*}}
\def\balignt#1\ealignt{\begin{talign}#1\end{talign}}

\let\originalleft\left
\let\originalright\right
\renewcommand{\left}{\mathopen{}\mathclose\bgroup\originalleft}
\renewcommand{\right}{\aftergroup\egroup\originalright}

\def\Gronwall{Gr\"onwall\xspace}
\def\Holder{H\"older\xspace}
\def\Ito{It\^o\xspace}
\def\Nystrom{Nystr\"om\xspace}
\def\Schatten{Sch\"atten\xspace}
\def\Matern{Mat\'ern\xspace}

\def\tinycitep*#1{{\tiny\citep*{#1}}}
\def\tinycitealt*#1{{\tiny\citealt*{#1}}}
\def\tinycite*#1{{\tiny\cite*{#1}}}
\def\smallcitep*#1{{\scriptsize\citep*{#1}}}
\def\smallcitealt*#1{{\scriptsize\citealt*{#1}}}
\def\smallcite*#1{{\scriptsize\cite*{#1}}}

\def\blue#1{\textcolor{blue}{{#1}}}
\def\green#1{\textcolor{green}{{#1}}}
\def\orange#1{\textcolor{orange}{{#1}}}
\def\purple#1{\textcolor{purple}{{#1}}}
\def\red#1{\textcolor{red}{{#1}}}
\def\teal#1{\textcolor{teal}{{#1}}}

\def\mbi#1{\boldsymbol{#1}} 
\def\mbf#1{\mathbf{#1}}
\def\mrm#1{\mathrm{#1}}
\def\tbf#1{\textbf{#1}}
\def\tsc#1{\textsc{#1}}

\def\mbiA{\mbi{A}}
\def\mbiB{\mbi{B}}
\def\mbiC{\mbi{C}}
\def\mbiDelta{\mbi{\Delta}}
\def\mbif{\mbi{f}}
\def\mbiF{\mbi{F}}
\def\mbih{\mbi{g}}
\def\mbiG{\mbi{G}}
\def\mbih{\mbi{h}}
\def\mbiH{\mbi{H}}
\def\mbiI{\mbi{I}}
\def\mbim{\mbi{m}}
\def\mbiP{\mbi{P}}
\def\mbiQ{\mbi{Q}}
\def\mbiR{\mbi{R}}
\def\mbiv{\mbi{v}}
\def\mbiV{\mbi{V}}
\def\mbiW{\mbi{W}}
\def\mbiX{\mbi{X}}
\def\mbiY{\mbi{Y}}
\def\mbiZ{\mbi{Z}}

\def\textsum{{\textstyle\sum}} 
\def\textprod{{\textstyle\prod}} 
\def\textbigcap{{\textstyle\bigcap}} 
\def\textbigcup{{\textstyle\bigcup}} 

\def\reals{\mathbb{R}} 
\def\integers{\mathbb{Z}} 
\def\rationals{\mathbb{Q}} 
\def\naturals{\mathbb{N}} 
\def\complex{\mathbb{C}} 

\def\what#1{\widehat{#1}}

\def\twovec#1#2{\left[\begin{array}{c}{#1} \\ {#2}\end{array}\right]}
\def\threevec#1#2#3{\left[\begin{array}{c}{#1} \\ {#2} \\ {#3} \end{array}\right]}
\def\nvec#1#2#3{\left[\begin{array}{c}{#1} \\ {#2} \\ \vdots \\ {#3}\end{array}\right]} 

\def\maxeig#1{\lambda_{\mathrm{max}}\left({#1}\right)}
\def\mineig#1{\lambda_{\mathrm{min}}\left({#1}\right)}

\def\Re{\operatorname{Re}} 
\def\indic#1{\mbb{I}\left[{#1}\right]} 
\def\logarg#1{\log\left({#1}\right)} 
\def\polylog{\operatorname{polylog}}
\def\maxarg#1{\max\left({#1}\right)} 
\def\minarg#1{\min\left({#1}\right)} 
\def\Earg#1{\E\left[{#1}\right]}
\def\Esub#1{\E_{#1}}
\def\Esubarg#1#2{\E_{#1}\left[{#2}\right]}
\def\bigO#1{\mathcal{O}\left(#1\right)} 
\def\littleO#1{o(#1)} 
\def\P{\mbb{P}} 
\def\Parg#1{\P\left({#1}\right)}
\def\Psubarg#1#2{\P_{#1}\left[{#2}\right]}
\def\Trarg#1{\Tr\left[{#1}\right]} 
\def\trarg#1{\tr\left[{#1}\right]} 
\def\Var{\mrm{Var}} 
\def\Vararg#1{\Var\left[{#1}\right]}
\def\Varsubarg#1#2{\Var_{#1}\left[{#2}\right]}
\def\Cov{\mrm{Cov}} 
\def\Covarg#1{\Cov\left[{#1}\right]}
\def\Covsubarg#1#2{\Cov_{#1}\left[{#2}\right]}
\def\Corr{\mrm{Corr}} 
\def\Corrarg#1{\Corr\left[{#1}\right]}
\def\Corrsubarg#1#2{\Corr_{#1}\left[{#2}\right]}
\newcommand{\info}[3][{}]{\mathbb{I}_{#1}\left({#2};{#3}\right)} 
\newcommand{\staticexp}[1]{\operatorname{exp}(#1)} 
\newcommand{\loglihood}[0]{\mathcal{L}} 


\providecommand{\arccos}{\mathop\mathrm{arccos}}
\providecommand{\dom}{\mathop\mathrm{dom}}
\providecommand{\diag}{\mathop\mathrm{diag}}
\providecommand{\tr}{\mathop\mathrm{tr}}
\providecommand{\card}{\mathop\mathrm{card}}
\providecommand{\sign}{\mathop\mathrm{sign}}
\providecommand{\conv}{\mathop\mathrm{conv}} 
\def\rank#1{\mathrm{rank}({#1})}
\def\supp#1{\mathrm{supp}({#1})}

\providecommand{\minimize}{\mathop\mathrm{minimize}}
\providecommand{\maximize}{\mathop\mathrm{maximize}}
\providecommand{\subjectto}{\mathop\mathrm{subject\;to}}

\def\openright#1#2{\left[{#1}, {#2}\right)}

\ifdefined\nonewproofenvironments\else
\ifdefined\ispres\else
 
\fi
\makeatletter
\@addtoreset{equation}{section}
\makeatother
\def\theequation{\thesection.\arabic{equation}}

\newcommand{\cmark}{\ding{51}}

\newcommand{\xmark}{\ding{55}}

\newcommand{\eq}[1]{\begin{align}#1\end{align}}
\newcommand{\eqn}[1]{\begin{align*}#1\end{align*}}
\newcommand{\Ex}[1]{\mathbb{E}\left[#1\right]}

\newcommand{\matt}[1]{{\textcolor{Maroon}{[Matt: #1]}}}
\newcommand{\kook}[1]{{\textcolor{blue}{[Kook: #1]}}}
\definecolor{OliveGreen}{rgb}{0,0.6,0}

\global\long\def\on#1{\operatorname{#1}}%

\global\long\def\bw{\mathsf{Ball\ walk}}%
\global\long\def\sw{\mathsf{Speedy\ walk}}%
\global\long\def\gw{\mathsf{Gaussian\ walk}}%
\global\long\def\ps{\mathsf{Proximal\ sampler}}%
\global\long\def\sps{\mathsf{Proximal\ Gaussian\ cooling}}%

\global\long\def\har{\mathsf{Hit\text{-}and\text{-}run}}%
\global\long\def\gc{\mathsf{Gaussian\ cooling}}%
\global\long\def\ino{\mathsf{\mathsf{In\text{-}and\text{-}Out}}}%

\global\long\def\O{\mathcal{O}}%
\global\long\def\Otilde{\widetilde{\mathcal{O}}}%
\global\long\def\Omtilde{\widetilde{\Omega}}%

\global\long\def\E{\mathbb{E}}%
\global\long\def\Z{\mathbb{Z}}%
\global\long\def\P{\mathbb{P}}%
\global\long\def\N{\mathbb{N}}%
\global\long\def\K{\mathcal{K}}%

\global\long\def\R{\mathbb{R}}%
\global\long\def\Rd{\mathbb{R}^{d}}%
\global\long\def\Rdd{\mathbb{R}^{d\times d}}%
\global\long\def\Rn{\mathbb{R}^{n}}%
\global\long\def\Rnn{\mathbb{R}^{n\times n}}%

\global\long\def\psd{\mathbb{S}_{+}^{d}}%
\global\long\def\pd{\mathbb{S}_{++}^{d}}%

\global\long\def\defeq{\stackrel{\mathrm{{\scriptscriptstyle def}}}{=}}%

\global\long\def\veps{\varepsilon}%
\global\long\def\lda{\lambda}%
\global\long\def\vphi{\varphi}%
\global\long\def\cpi{C_{\mathsf{PI}}}%

\global\long\def\half{\frac{1}{2}}%
\global\long\def\nhalf{\nicefrac{1}{2}}%
\global\long\def\texthalf{{\textstyle \frac{1}{2}}}%

\global\long\def\ind{\mathds{1}}%
\global\long\def\op{\mathsf{op}}%

\global\long\def\chooses#1#2{_{#1}C_{#2}}%

\global\long\def\vol{\on{vol}}%

\global\long\def\law{\on{law}}%

\global\long\def\tr{\on{tr}}%

\global\long\def\diag{\on{diag}}%

\global\long\def\Diag{\on{Diag}}%

\global\long\def\inter{\on{int}}%

\global\long\def\esssup{\on{ess\,sup}}%

\global\long\def\polylog{\on{polylog}}%

\global\long\def\var{\on{var}}%

\global\long\def\e{\mathrm{e}}%

\global\long\def\id{\mathrm{id}}%

\global\long\def\spanning{\on{span}}%

\global\long\def\rows{\on{row}}%

\global\long\def\cols{\on{col}}%

\global\long\def\rank{\on{rank}}%

\global\long\def\T{\mathsf{T}}%

\global\long\def\bs#1{\boldsymbol{#1}}%

\global\long\def\eu#1{\EuScript{#1}}%

\global\long\def\mb#1{\mathbf{#1}}%

\global\long\def\mbb#1{\mathbb{#1}}%

\global\long\def\mc#1{\mathcal{#1}}%

\global\long\def\mf#1{\mathfrak{#1}}%

\global\long\def\ms#1{\mathscr{#1}}%

\global\long\def\mss#1{\mathsf{#1}}%

\global\long\def\msf#1{\mathsf{#1}}%

\global\long\def\textint{{\textstyle \int}}%
\global\long\def\Dd{\mathrm{D}}%
\global\long\def\D{\mathrm{d}}%
\global\long\def\grad{\nabla}%
 
\global\long\def\hess{\nabla^{2}}%
 
\global\long\def\lapl{\triangle}%
 
\global\long\def\deriv#1#2{\frac{\D#1}{\D#2}}%
 
\global\long\def\pderiv#1#2{\frac{\partial#1}{\partial#2}}%
 
\global\long\def\de{\partial}%
\global\long\def\lagrange{\mathcal{L}}%
\global\long\def\Div{\on{div}}%

\global\long\def\Gsn{\mathcal{N}}%
 
\global\long\def\BeP{\textnormal{BeP}}%
 
\global\long\def\Ber{\textnormal{Ber}}%
 
\global\long\def\Bern{\textnormal{Bern}}%
 
\global\long\def\Bet{\textnormal{Beta}}%
 
\global\long\def\Beta{\textnormal{Beta}}%
 
\global\long\def\Bin{\textnormal{Bin}}%
 
\global\long\def\BP{\textnormal{BP}}%
 
\global\long\def\Dir{\textnormal{Dir}}%
 
\global\long\def\DP{\textnormal{DP}}%
 
\global\long\def\Expo{\textnormal{Expo}}%
 
\global\long\def\Gam{\textnormal{Gamma}}%
 
\global\long\def\GEM{\textnormal{GEM}}%
 
\global\long\def\HypGeo{\textnormal{HypGeo}}%
 
\global\long\def\Mult{\textnormal{Mult}}%
 
\global\long\def\NegMult{\textnormal{NegMult}}%
 
\global\long\def\Poi{\textnormal{Poi}}%
 
\global\long\def\Pois{\textnormal{Pois}}%
 
\global\long\def\Unif{\textnormal{Unif}}%

\global\long\def\bpar#1{\bigl(#1\bigr)}%
\global\long\def\Bpar#1{\Bigl(#1\Bigr)}%

\global\long\def\snorm#1{\|#1\|}%
\global\long\def\bnorm#1{\bigl\Vert#1\bigr\Vert}%
\global\long\def\Bnorm#1{\Bigl\Vert#1\Bigr\Vert}%

\global\long\def\sbrack#1{[#1]}%
\global\long\def\bbrack#1{\bigl[#1\bigr]}%
\global\long\def\Bbrack#1{\Bigl[#1\Bigr]}%

\global\long\def\sbrace#1{\{#1\}}%
\global\long\def\bbrace#1{\bigl\{#1\bigr\}}%
\global\long\def\Bbrace#1{\Bigl\{#1\Bigr\}}%

\global\long\def\Abs#1{|#1|}%
\global\long\def\Par#1{\left(#1\right)}%
\global\long\def\Brack#1{\left[#1\right]}%
\global\long\def\Brace#1{\left\{  #1\right\}  }%

\global\long\def\inner#1{\langle#1\rangle}%
 
\global\long\def\binner#1#2{\left\langle {#1},{#2}\right\rangle }%

\global\long\def\norm#1{\|#1\|}%
\global\long\def\onenorm#1{\norm{#1}_{1}}%
\global\long\def\twonorm#1{\norm{#1}_{2}}%
\global\long\def\infnorm#1{\norm{#1}_{\infty}}%
\global\long\def\fronorm#1{\norm{#1}_{\text{F}}}%
\global\long\def\nucnorm#1{\norm{#1}_{*}}%
\global\long\def\staticnorm#1{\|#1\|}%
\global\long\def\statictwonorm#1{\staticnorm{#1}_{2}}%

\global\long\def\mmid{\mathbin{\|}}%

\global\long\def\otilde#1{\widetilde{\mc O}(#1)}%
\global\long\def\wtilde{\widetilde{W}}%
\global\long\def\wt#1{\widetilde{#1}}%

\global\long\def\KL{\msf{KL}}%
\global\long\def\dtv{d_{\textrm{\textup{TV}}}}%
\global\long\def\FI{\msf{FI}}%
\global\long\def\tv{\msf{TV}}%

\global\long\def\cov{\mathrm{Cov}}%
\global\long\def\PI{\msf{PI}}%

\global\long\def\cred#1{\textcolor{red}{#1}}%
\global\long\def\cblue#1{\textcolor{blue}{#1}}%
\global\long\def\cgreen#1{\textcolor{green}{#1}}%
\global\long\def\ccyan#1{\textcolor{cyan}{#1}}%

\global\long\def\iff{\Leftrightarrow}%
 
\global\long\def\textfrac#1#2{{\textstyle \frac{#1}{#2}}}%

\newcommand{\NA}{N_{\mrm{ker}}}
\newcommand{\NB}{N_{\mrm{tot}}}

\newcommand\deq{\coloneqq}

\newcommand\fp{\mathfrak{p}}
\title{Covariance estimation using Markov chain Monte Carlo}
\author{
Yunbum Kook\thanks{
  School of Computer Science,
  Georgia Institute of Technology, \texttt{yb.kook@gatech.edu}
} \ \ \ \ \ \ \ \ \ \ \ \ \ \ \ \ \ 
Matthew S.\ Zhang\thanks{
  Dept. of Computer Science,
  University of Toronto, and Vector Institute, \texttt{matthew.zhang@mail.utoronto.ca}
}
}

\maketitle
\begin{abstract}
We investigate the complexity of covariance matrix estimation for Gibbs distributions based on dependent samples from a Markov chain. We show that when $\pi$ satisfies a Poincar\'e inequality and the chain possesses a spectral gap, we can achieve similar sample complexity using MCMC as compared to an estimator constructed using i.i.d. samples, with potentially much better query complexity. As an application of our methods, we show improvements for the query complexity in both constrained and unconstrained settings for concrete instances of MCMC. In particular, we provide guarantees regarding isotropic rounding procedures for sampling uniformly on convex bodies.
\end{abstract}

\section{Introduction}

We focus on the problem of mean and covariance estimation using Markov chain Monte Carlo (MCMC) methods. One setting which often arises in practice is when we want to approximately estimate these statistics for a distribution $\pi \propto \exp(-V)$, where we are only given query access to $V: \R^d \to \R$ and its gradient. 
For example, this occurs when $\pi$ is given by a Bayesian posterior for a large dataset. As it is impossible in general to obtain the normalizing constant for the measure, one can only hope to obtain approximate samples from $\pi$ via a MCMC procedure. 

Precisely, suppose we are given some $\pi$-stationary Markov kernel $P: \R^d \times \mc B(\R^d) \to \R$, with $P(A\,|\,x) \deq P(x, A)$ and an initial  tractable probability measure $\pi_0$.
Then, the MCMC scheme corresponding to $(\pi_0, P)$ generates the sequence $\{X_k\}_{k \in [N]}$ through the procedure
\begin{equation}\label{eq:mcmc}
    X_0 \sim \pi_0\,, \qquad \qquad X_k \sim P(\cdot\,\vert\, X_{k-1}) \qquad \text{for all } k = 1, \dotsc, N\,.
\end{equation}
The cost for this algorithm is quantified by the total number $\NB$ of queries; we are interested in its dependence on $d$ and other problem parameters to be defined later.

While the standard theory for covariance estimation assumes \emph{independent and identically distributed} (i.i.d.) samples, in practice statisticians take them from a sequence generated by some Markov chain (despite the loss of the i.i.d. property); they use the iterates $\{X_{N_{\mrm{burn}} + k}\}_{k\leq N_{\mrm{samp}}}$ to construct their estimates, with the ``burn-in'' time $N_{\mrm{burn}}$ needed to generate a good initial sample. From this, the next $N_{\mrm{samp}}$ samples are collected and used for the statistical estimation process, with $N_{\mrm{samp}}$ chosen so that we can construct an estimator that is suitably concentrated around the quantity of interest.
Given that in most circumstances the cost of generating a \emph{warm}-start (i.e., finding a $\pi_0$ that is $\mc O(1)$-close to $\pi)$ dominates that of bringing $\pi_0$ to $\pi$ within desired accuracy (i.e.\ $N_{\mrm{burn}} \gg N_{\mrm{samp}}$),  
it is much more cost-effective to continue evolving the Markov chain from the previous sample, rather than restarting from ground up.

However, na\"ively applying the standard i.i.d. concentration analysis for the resulting estimands constructed with \emph{dependent} samples does not lead to a useful bound. Indeed, without leveraging an additional structural property for $P$ such as spectral gap, then the estimand could never concentrate even as $N_{\mrm{samp}} \to \infty$ (for instance when $P = \operatorname{id}$).

Recent years have seen a wealth of results establishing convergence guarantees for various MCMC algorithms; see~\cite{chewi2024log} for a thorough exposition. These results generally show the exponential convergence\footnote{Up to a numerical error term induced by the discretization of the procedure.} of $\law(X_k)$ to $\pi$ in various strong notions of ``distance'', so long as $\pi$ satisfies a suitable \emph{functional inequality}. The present work operates in the setting where $\pi$ satisfies a \textbf{Poincar\'e inequality}~\eqref{eq:poincare} with constant $\cpi$, and also assumes that $P$ possesses a \textbf{spectral gap} of $\lambda > 0$. See \S\ref{ssec:background} for precise definitions of these two concepts; intuitively, the Poincar\'e inequality says that the tail of $\pi$ should decay on the order of $\exp(-\norm{x}/\sqrt{\cpi})$, while the spectral gap guarantees that $\law(X_k)$ converges exponentially to $\pi$ in an $L^2$ sense. 

While these two assumptions may appear unrelated at a first glance, for many settings they are actually equivalent. The model example is the Langevin diffusion with target distribution $\pi$, a stochastic process whose discretizations form the bedrock for many sampling algorithms. In this case, spectral gap of the diffusion run to a time $h$ is \emph{equivalent} to a Poincar\'e inequality on its target.
Another case is the \emph{proximal sampler} \cite{lee2021structured,chen2022improved}.
Running this sampler with step size $h$ ensures that $\chi^{2}(\mu P^k\mmid\pi)\leq (1+h/\cpi(\pi))^{-2k} \chi^{2}(\mu \mmid\pi)$. Iterating this $h^{-1}\cpi(\pi)$ times consecutively will generate a Markov kernel with spectral gap arbitrarily close to $1$. 

Although there are other known approaches for handling dependence of samples (e.g., $\alpha$-mixing, see \S\ref{sec:related-work}), this pair of assumptions will turn out to be handy since they capture many settings of interest. Firstly, the Poincar\'e inequality holds across a broad class of distributions. In particular, a celebrated recent result in functional analysis shows that~\eqref{eq:poincare} holds with constant $\O(\norm{\text{Cov}(\pi)}_{\op}\log d)$ for \emph{any} log-concave distribution~\cite{klartag2023logarithmic}, which includes the uniform distribution over a convex body. Secondly, many standard MCMC algorithms~\cite{lovasz2006hit, chen2022improved, kook2024and} yield kernels with spectral gaps under~\eqref{eq:poincare}. 
Thus, determining the sample and query complexity needed for accurate statistical estimation under these assumptions is of paramount importance.

Encouragingly, \cite{neeman2023concentration} shows that under suitable assumptions on spectral gap, the sum of a Markovian sequence of matrices will satisfy a Bernstein-type inequality. Motivated by these efforts, we investigate the following research question:
\begin{center}
    \textbf{Question:} What is the sample and query complexity of covariance estimation for a probability measure $\pi$ which satisfies a Poincar\'e inequality~\eqref{eq:poincare}?
\end{center}

\subsection{Results}
\paragraph{Covariance estimation via MCMC.}
We demonstrate that MCMC methods for distributions satisfying~\eqref{eq:poincare} reduce the query complexity by a significant factor, compared to an estimator constructed using i.i.d. samples. We summarize our main result in the theorem below.  A similar guarantee also holds for the mean vector $\overline X \deq N^{-1} \sum_{i=1}^N X_i \approx \E_\pi X$.
\begin{thm}[{Informal; see Theorem~\ref{thm:add-cov-Poincare}}]
Assume that $\pi$ satisfies a Poincar\'e inequality and $P$ has a spectral gap.
For $\varepsilon>0$ and $\delta\in (0, d)$, the estimator $\overline \Sigma \deq \frac{1}{N} \sum_{i=1}^N (X_i - \overline X)^{\otimes 2}$ satisfies with high probability that
$
\Abs{\overline{\Sigma}-\textup{Cov}(\pi)}\preceq\veps\Sigma+\delta I_{d}\,,
$
so long as $N= \Otilde(\frac{\tr\Sigma+\cpi(\pi)}{\veps\delta})$.
\end{thm}
We also provide a multiplicative form in Corollary~\ref{cor:mul-cov-Poincare}. To the best of our knowledge, we are the first work to give a theory of covariance estimation using MCMC under the standard assumptions of a Poincar\'e inequality and a spectral gap. 

\paragraph{Applications.}
To illustrate the broad applicability of our results, we provide several domains where our bounds can be leveraged for significant improvements in query complexity.

The first occurs when sampling uniformly from a convex body $\K \subset \Rd$ (i.e., $\pi \propto \one_{\mc K}$). This problem initiated the quantitative study of covariance estimation in theoretical computer science (see \S\ref{sec:related-work}).
We demonstrate that, when $\mc K$ is well-rounded, $\E_\pi [\norm{\cdot}^2] \lesssim d$, we can estimate the covariance up to a constant factor with probability $1-\mc O(\nicefrac{1}{d})$, using roughly $d^3$ oracle queries in expectation (Lemma~\ref{lem:cov-estimation-app1}). This contrasts with the best known complexity of roughly $d^4$ in the setting using i.i.d. samples. Using this, we can therefore convert a well-rounded convex body to a near-isotropic one with the same complexity (Lemma~\ref{lem:rounding-final}). This is a key ingredient in the \emph{isotropic rounding} procedure of~\cite{jia2021reducing}, needed for the state-of-the-art complexity in uniform constrained sampling. 
Alongside this, we provide a simpler proof of the rounding algorithm therein (see \S\ref{sec:app-iso-rounding}).

Secondly, we show that for unconstrained sampling $\pi \propto \exp(-V)$, when $\pi$ satisfies~\eqref{eq:poincare} with constant $\alpha^{-1}$ and $V$ is $\beta$-smooth, there exists an algorithm with query complexity $\Otilde(\kappa d^{3/2}/\varepsilon^2)$ for $\kappa \deq \beta/\alpha$ which achieves a covariance estimate $\lVert\overline \Sigma - \Sigma\rVert \leq \varepsilon \lVert\Sigma\rVert$. If instead one were to use i.i.d. samples, then the best known complexity is $\O(\kappa d^{5/2}/\varepsilon^2)$, which is worse by a dimensional factor. See Theorem~\ref{thm:unconstrained-covar} for more details.

\subsection{Related work}\label{sec:related-work}
\paragraph{Covariance estimation.} 
Covariance estimation has long been studied in statistics. We concentrate on the following setting: \emph{given random i.i.d. samples $(X_i)_{i \leq N}$ from a centered distribution $\pi$ with $\Sigma:=\textup{Cov}(\pi)$, bound the number $N$ of samples for which the following holds with high probability}
\begin{equation}\label{eq:cov-multi}
    \Bnorm{\frac{1}{N}\sum_{i=1}^N X_i^{\otimes 2} - \Sigma}\leq \frac{1}{10} \norm{\Sigma}\,.
\end{equation}
In this framework, the problem was first addressed by~\cite{kannan1997random}, obtaining $N=\O(d^2)$ for a uniform distribution over a convex body. Their analysis occurred in the context of isotropic rounding of a convex body, which was one of the subroutines in their volume-computation algorithm. 
This was improved for general distributions by~\cite{bourgain1996random, rudelson1999random}, and by~\cite{paouris2006concentration} for a convex body. The main technique used by all of these works was a noncommutative moment inequality. Then, the seminal work of~\cite{adamczak2010quantitative} showed via a chaining argument that $N=\O(d)$ for any log-concave distribution. They also showed that this holds for a sub-exponential distribution if additionally $\norm{X}\lesssim d^{1/2}$ (see Remark 4.11 therein). 
This bound \emph{without logarithmic overhead} was further extended by~\cite{srivastava2013covariance} under a bounded moment condition. 

Another line of work follows from the matrix Laplace transform method, pioneered by~\cite{ahlswede2002strong} and further developed in~\cite{oliveira2009concentration, vershynin2012introduction, tropp2012user}. The eventual bound obtained is also $N = \mc O(d \log d)$, where logarithmic oversampling is in general unavoidable. For a more comprehensive survey of the literature, see the following monographs~\cite{vershynin2012introduction, vershynin2018high, rigollet2023high}.

\paragraph{Concentration inequalities under dependence.}
Concentration inequalities such as the matrix Bernstein inequality serves as a general arsenal for establishing the sample complexity of this type of problem in the setting of \emph{i.i.d. data}. We refer readers to \cite{tropp2015introduction} for a comprehensive overview of this topic.
By contrast, the setting where $\{X_1, \ldots, X_n\}$ are correlated has been comparatively understudied. Beginning with~\cite{gillman1993hidden}, several works~\cite{lezaud1998chernoff, paulin2015concentration, jiang2018bernstein, fan2021hoeffding} have proven concentration inequalities for (scalar) functions of iterates arising from a Markov process.

Matrix concentration results under Markovian assumptions have only recently been studied~\cite{garg2018matrix, qiu2020matrix}. Most relevant to our work is that of~\cite{neeman2023concentration}, which establishes a Bernstein inequality in the Frobenius norm, assuming that each sample is bounded and that the Markov chain has a spectral gap. We also note that other literature has shown concentration for correlated iterates under conditions other than a spectral gap~\cite{mackey2014matrix, paulin2016efron}. In particular, Bernstein inequalities have been established under other assumptions such as the $\beta$-mixing or $\tau$-mixing of the chain. Under these properties, the representative works~\cite{merlevede2009bernstein, merlevede2011bernstein, banna2016bernstein} show a Bernstein-like matrix concentration inequality, analogous to~\cite{neeman2023concentration}. However, in our applications it is difficult to verify these conditions, unlike the spectral gap condition considered in our work.

Although matrix Bernstein inequalities are indeed very strong, they need to be combined with a tail estimate in order to be usable in our setting. \cite{adamczak2010quantitative} shows that this is possible under a log-concavity assumption of the target distribution. However, their argument cannot be applied out-of-the-box to Markovian iterates, and requires some additional adjustments to remain valid.

\paragraph{Sampling under functional inequalities.}
The relationship between Markov chain methods and spectral gaps dates back to the early theory of MCMC~\cite{mihail1989conductance, fill1991eigenvalue}, and has been thoroughly characterized in discrete-space settings.

In continuous state-space, however, the problem is complicated by the need to ensure implementability of the resulting algorithm. 
Early results characterized the convergence under log-concavity or log-Sobolev conditions~\cite{dalalyan2012sparse, dalalyan2017further, durmus2019analysis, vempala2019rapid}. While the continuous-time convergence of many idealized processes can be derived easily from Poincar\'e-type inequalities~\cite{bakry2014analysis}, only recently were results proven under~\eqref{eq:poincare} for implementable samplers such as the Langevin Monte Carlo algorithm or the proximal sampler~\cite{lehec2023langevin, chewi2021analysis, chen2022improved, altschuler2024faster}. These works demonstrate that when $\pi$ satisfies~\eqref{eq:poincare}, the law of the iterates generated by the sampler achieves $\varepsilon$-accuracy in $\chi^2(\cdot \mmid\pi)$ with mild complexity, and that the proximal sampler actually contracts in said divergence for any input measure. This is equivalent to a spectral gap for the proximal sampler. As a result, we focus our attention on the proximal sampler, since merely guaranteeing $\varepsilon$-accuracy in $\chi^2$-divergence is not sufficient to establish the concentration of an estimand constructed from Markovian iterates.

For constrained targets such as uniform distributions over convex bodies, $\bw$~\cite{lovasz1993random, kannan1997random} and $\har$~\cite{smith1984efficient, lovasz1999hit} were shown to have mixing times depending on the Cheeger isoperimetric constant of the target. Only recently, \cite{kook2024and, kook2024renyi} proposed a sampler with R\'enyi-infinity guarantees which directly quantify the dependence on the Poincar\'e constant; this in turn depends on the degree of \emph{isotropy} of the target, i.e., how close its covariance matrix is to the identity.

A related line of research~\cite{lovasz2006simulated, cousins2018gaussian, jia2021reducing} discovered a procedure which could place any convex body in isotropic position. They iteratively estimate the covariance matrix of subsets of the convex body, and then sample from it with only $\Otilde(d^{3.5})$ complexity. To achieve their complexity, they rely on a Markovian covariance estimator subroutine, which we also study in the present work.

\subsection{Organization}
The paper is organized as follows. In \S\ref{sec:prelim}, we provide some preliminaries needed for our main result. In \S\ref{sec:covar_est}, we state our main bound on the concentration of covariance matrices generated by a Markov chain, and present a brief proof sketch. \S\ref{sec:applications} then gives our two primary applications for covariance estimation, and includes some additional details for the isotropic rounding algorithm. \S\ref{app:proofs-main} and \S\ref{app:proofs-app} contain the proofs for the results of the previous sections. We conclude by summarizing future directions along this line of research.

\section{Preliminaries}\label{sec:prelim}

\subsection{Notation}

We use $\norm{\cdot}$ for the $\ell_2$-norm on $\Rd$, and $\norm{\cdot}$ for the operator norm on $\R^{d \times d}$, while $\norm{\cdot}_{\mrm{F}}$ indicates the Frobenius norm. For a vector $u\in \Rd$, we denote its outer product by $u^{\otimes 2}:= uu^\T$. Next,  $a=\mc O(b)$ or $a\lesssim b$ signifies that $a\leq cb$ for an absolute constant $c>0$. Similarly, $a\gtrsim b,a=\Omega(b)$ signifies $a\geq cb$, while $a=\Theta(b)$ signifies $a\lesssim b,b\lesssim a$ simultaneously. We also use $a=\Otilde(b)$ to denote $a=\mc O(b)$ up to a polylogarithmic factor. 
When comparing matrices, we write $|A|\preceq B$ for two matrices $A,B \in \R^{d\times d}$ to denote $-B\preceq A\preceq B$. We conflate a measure and its density throughout the work where there is no confusion. Lastly, $\mc P(\Rd)$ indicates the set of probability measures over $\Rd$. 

Since we are primarily concerned with the high-dimensional complexity of the algorithm, many of our results involve algebraic simplifications that are valid only when the dimension $d \geq 200$ is sufficiently large. In quantifying failure probabilities for our algorithm, we introduce a constant $\fp$ which may change from line to line, but is always some universal constant that does not depend on $d, \veps$ or any other parameter of interest.

\subsection{Background}\label{ssec:background}

We now formalize key notions needed for our main results.
\begin{defn}\label{def:poincare}
We say that a probability measure $\pi$ on $\R^{d}$ satisfies a \emph{Poincar\'e inequality} (PI) with parameter $\cpi(\pi) \geq 0$ if for all smooth functions $f:\R^{d}\to\R$, 
\begin{equation}
\var_{\pi}f\leq C_{\PI}(\pi)\,\E_{\pi}[\norm{\nabla f}^{2}]\,,\tag{PI}\label{eq:poincare}
\end{equation}
 where $\var_{\pi}f\deq\E_{\pi}[\Abs{f-\E_{\pi}f}^{2}]$. 
\end{defn}
In particular, \eqref{eq:poincare} implies that any $1$-Lipschitz function $f$ concentrates around its mean $\E_{\pi}f$ in a sub-exponential manner:
\begin{equation}\label{eq:poincare-Lip-con}
\P_\pi(f-\E_{\pi}f\geq t)\leq3\exp\bpar{-\frac{t}{\sqrt{C_{\PI}(\pi)}}}\qquad\text{for all }t\geq0\,.    
\end{equation}

\begin{defn}\label{def:spectral}
    Let $P: \R^d \times \mc B(\R^d) \to \R$ be a Markov kernel, i.e., $P(\cdot\,|\,x) \in \mc P(\R^d)$ for every $x \in \R^d$. 
    For the stationary measure $\pi \in \mc P(\R^d)$ of the Markov kernel, $P$ is said to have a \emph{spectral gap} $\lambda \in (0, 1]$ if for every $f: \R^d \to \R$ with $f \in L^2(\pi)$ and $\E_\pi f = 0$, we have
    \[
        \norm{P f}_{L^2(\pi)}^2 \deq \E_{X\sim\pi} \bbrack{\bigl(\E_{Y\sim P(\cdot|X)}[f(Y)]\bigr)^2}\leq (1-\lambda)^2 \norm{f}_{L^2(\pi)}^2\,.
    \]
    A kernel $P$ is \emph{reversible} if $\inner{Pf, g}_{L^2(\pi)} = \inner{f, Pg}_{L^2(\pi)}$ for all $f, g \in L^2(\pi)$.
\end{defn}
We say a sequence $(X_0, X_1, \ldots)$ is \emph{driven by} $P$ with initial distribution $\pi_0$ if it is generated by the procedure in~\eqref{eq:mcmc}. The following metrics\footnote{Although not all of these are proper metrics, we still refer to them as such in sense of ``performance metric''.} between probability measures will be useful throughout our work, particularly for the applications.
\begin{defn}[Probability divergences] \label{def:p-dist} 
For $\mu,\nu \in \mc P(\Rd)$, the \emph{$\chi^{q}$-divergence} is defined by
\[
\chi^q(\mu\mmid\nu):=\int \bpar{\frac{\D\mu}{\D\nu}}^q\,\D\nu-1\quad\text{ if }\mu\ll\nu, \text{ and }\infty \text { otherwise.}
\]
In particular, $\chi^2(\mu \mmid \nu) = \textint (\tfrac{\D \mu}{\D \nu} - 1)^2 \, \D \nu$.
The \emph{$q$-R\'enyi divergence} is given by
\[
\eu R_{q}(\mu\mmid\nu):=\tfrac{1}{q-1}\log\bpar{\chi^{q}(\mu\mmid\nu)+1}\,.
\]
Additionally, the \emph{total variation} distance between the two measures is given by
\[
\norm{\mu-\nu}_{\msf{TV}} :=\sup_{B\in\mc F}\Abs{\mu(B)-\nu(B)}\,,
\]
where $\mc F$ is the set of all measurable subsets of $\R^{d}$.
\end{defn}

\section{Covariance estimation using dependent samples}\label{sec:covar_est}
We state our primary assumption below.
\begin{assumption}
\label{as:markov} $P$ is a reversible Markov kernel on $\Rd$ with stationary distribution $\pi$ and spectral gap $\lda$ (Definition~\ref{def:spectral}).
Let $(X_{1},\dotsc,X_{N})$ be a sequence driven by $P$ with initial distribution $\pi_0 = \pi$.
\end{assumption}
Note that the choice of $\pi_0 = \pi$ can be removed at the cost of a multiplicative factor $1+\chi^2(\pi_0 \mmid \pi)$ in our failure probabilities (see Lemma~\ref{lem:change-measure}).

The following result will be instrumental in establishing our covariance bounds.
\begin{thm}[{\cite{neeman2023concentration}, Theorem 2.2}] \label{thm:dependent-bernstein}
Under Assumption~\ref{as:markov}, suppose that $F_{i}:\R^d \to\R^{d\times d}$,\footnote{This is valid even for arbitrary input state spaces, but this will not concern us in this work.}
for $i\in[N]$ is a sequence of functions each mapping to real symmetric
$d\times d$ matrix satisfying $\E_{\pi}[F_{i}(\cdot)]=0$, $\norm{\E_\pi[F_{i}^2(\cdot)]}\leq\mc V_{i}$,
and $\sup_{\R^d}\norm{F_{i}(\cdot)}\leq\mc M$. For
$\sigma^{2}:=\sum_{i=1}^{N}\mc V_{i}$, it holds that 
\[
\P\Bpar{\Bnorm{\sum_{i=1}^{N}F_{i}(X_{i})}_{\textup{F}}\geq t}\leq d^{2-\pi/4}\exp\Bpar{-\frac{t^{2}/(32/\pi^{2})}{\alpha(\lda)\,\sigma^{2}+\beta(\lda)\,\mc Mt}}\,,
\]
where $\alpha(\lda)=\frac{2-\lda}{\lambda}$ and $\beta(\lda)=\frac{8/\pi}{\lambda}$.
\end{thm}
When $\lambda$ is some absolute constant bounded away from $0$, the quantity in the exponential is $\mc{O}(\min(t^2/\sigma^2, t/\mc M))$, i.e., quadratic for small $t$ and linear for large $t$.

As a corollary, we can deduce similar guarantees for vectors/non-square matrices (see \S\ref{app:mtx-Bern}).
\begin{cor}\label{cor:vec-bern-spectral} 
Under Assumption~\ref{as:markov}, suppose that $\E_\pi X_i = 0$, $\E_\pi[\norm{X_{i}}^{2}]\leq\mc V_{i}$, and $\norm{X_{i}}\leq\mc M$ almost surely for $i \in [N]$. Then, denoting $\sigma^{2}=\sum_{i=1}^{N}\mc V_{i}$ and where $\alpha, \beta$ are as in Theorem~\ref{thm:dependent-bernstein}, the following holds for all $t\geq0$, 
\[
\P\Bpar{\Bnorm{\sum_{i=1}^{N}X_{i}}\geq t}\leq(d+1)^{2-\pi/4}\exp\bpar{-\frac{t^{2}/(32/\pi^{2})}{\alpha(\lda)\,\sigma^{2}+\beta(\lda)\,\mc Mt}}\,.
\]
\end{cor}

\subsection{Main results}

Suppose our Markov chain is nice enough that the spectral gap is $\lambda \geq 0.99$, which can always be done by composing a sufficient number of iterations of a more tractable chain with smaller spectral gap. Then our estimators for the mean and covariance are respectively
\[
\overline{X}:=\frac{1}{N}\sum_{i=1}^{N}X_{i}\,,\qquad\overline{\Sigma}:=\frac{1}{N}\sum_{i=1}^{N}(X_{i}-\overline{X})^{\otimes2}=\frac{1}{N}\sum_{i=1}^{N}X_{i}^{\otimes2}-\overline{X}^{\otimes2}\,.
\]
We state our covariance estimation results in both additive and multiplicative forms. 

\begin{thm}[Additive form] \label{thm:add-cov-Poincare} 
Under Assumption~\ref{as:markov}, if $\pi$ also satisfies~\eqref{eq:poincare} and $P$ has a spectral gap $\lda\geq0.99$, the covariance estimator $\overline{\Sigma}=\frac{1}{N}\sum_{i=1}^{N}(X_{i}-\overline{X})^{\otimes2}$ satisfies that for any $\veps>0$ and $\delta\in(0,d)$, with probability at least $1-\nicefrac{\fp}{d}$,
\[
\Abs{\overline{\Sigma}-\Sigma} \preceq\veps\Sigma+\delta I_{d}\,,
\]
so long as $N\asymp\frac{\tr\Sigma+\cpi(\pi)}{\veps\delta}\log^{2}\frac{d\max(d,\cpi(\pi)+\tr\Sigma)}{\delta}\log^{2}d$.
\end{thm}
One can establish multiplicative and spectral forms as well.
\begin{thm}[Multiplicative form]\label{thm:mul-cov-Poincare}
    In the setting of Theorem~\ref{thm:add-cov-Poincare}, $\overline{\Sigma}$ satisfies that for any $\veps\in(0,1)$, with probability at least $1-\nicefrac{\fp}{d}$,
\[
\norm{\overline{\Sigma}-\Sigma}\leq \veps \norm{\Sigma}\,
\]
so long as $N\asymp\frac{d}{\veps^2}\,\frac{\tr\Sigma+\cpi(\pi)}{\tr\Sigma}\log^{2}\frac{d\,(\cpi(\pi)+\tr\Sigma)}{\veps \norm{\Sigma}}\log^{2}d$.
\end{thm}
\begin{cor}[Multiplicative form; spectral] \label{cor:mul-cov-Poincare} In the setting of Theorem~\ref{thm:add-cov-Poincare}, $\overline{\Sigma}$ satisfies for any $\veps\in(0, 1)$, with probability at least $1-\nicefrac{\fp}{d}$,
\[
\Abs{\overline{\Sigma}-\Sigma}\preceq\veps\Sigma\,,
\]
so long as $N\asymp\frac{d+\cpi(\nu)}{\veps^{2}}\log^{2}\frac{d(d+\cpi(\nu))}{\veps}\log^{2}d$, where $\nu:=(\Sigma^{-1/2})_{\#}\pi$.
\end{cor}

\begin{rem}[Poincar\'e constant of $\nu$] \label{rem:PI-pushforward}
For $\nu=(\Sigma^{-1/2})_{\#}\pi$, it holds in general that $\cpi(\nu)\leq\lda_{1}^{-2}\cpi(\pi)$ for the smallest eigenvalue $\lda_{1}$ of $\Sigma$. In particular, it is well-known by \cite{klartag2023logarithmic} that a log-concave distribution $\pi$ satisfies 
\[
\norm{\Sigma}_{\op}\le\cpi(\pi)\lesssim\norm{\Sigma}_{\op}\log d\,.
\]
\end{rem}
\begin{rem}[Independent samples]\label{rmk:independent}
    As a corollary, we can obtain an analogous result when the samples are independent by considering the case where the kernel $P(\cdot\,|\,x) = \pi$ identically. In this case, the spectral gap is $1$, and the previous results can apply without alteration. As far as we can tell, this is the first time that such a result has been explicitly presented for distributions satisfying~\eqref{eq:poincare}, which may be of independent interest.
\end{rem}

\paragraph{Proof sketch}

We sketch a proof for Theorem~\ref{thm:add-cov-Poincare} below, deferring the detailed analysis to \S\ref{app:proofs-main}. In our analysis, we show that the error can be divided into two terms
\[
\overline{\Sigma}-\Sigma =\frac{1}{N}\sum_{i=1}^{N}\bpar{(X_{i}-\mu)^{\otimes2}-\Sigma}-(\overline{X}-\mu)^{\otimes2}\,.
\]
The techniques we use for the mean and covariance error terms will be similar, so we shall concentrate on the argument for the first term. We split this into three sub-terms, similar to the argument in~\cite{adamczak2010quantitative}, where $B \subseteq \R^{N \times d}$ will be some suitably ``nice'' set,
\[
\frac{1}{N}\sum X_{i}^{\otimes2}-\Sigma=\underbrace{\frac{1}{N}\sum_{i=1}^N(X_{i}^{\otimes2}\ind_{B}-\E[X_{i}^{\otimes2}\ind_{B}])}_{\eqqcolon\msf A_1}+\underbrace{\frac{1}{N}\sum_{i=1}^N X_{i}^{\otimes2}\ind_{B^{c}}}_{\eqqcolon\msf A_2}-\underbrace{\E[X^{\otimes2}\ind_{B^{c}}]}_{\eqqcolon\msf A_3}\,.
\]
Since the first term relates to the concentration of a bounded matrix around its mean, it can be handled using Theorem~\ref{thm:dependent-bernstein} after some detailed calculations (Lemma~\ref{lem:covar_termA}). As for the rest, we would like to use a concentration inequality under~\eqref{eq:poincare} to ensure that $\msf A_2, \msf A_3$ are small. While the argument of~\cite{adamczak2010quantitative} uses independence of the iterates to establish an exponential tail decay of $\norm{X}$, we are only able to establish a polynomial decay of $\norm{X}$ under Markovian assumptions (Lemma~\ref{lem:covar_termBC}). This is sufficient for Theorem~\ref{thm:add-cov-Poincare}. Nonetheless, we also show a slightly improved tail bound in \S\ref{app:finer-tail}, and it may be possible to sharpen this in the future.

\section{Applications}\label{sec:applications}

\subsection{Isotropic rounding via uniform sampling}\label{sec:app-iso-rounding}

As a first application, we consider a version of the ``sampling-and-rounding'' scheme used in \cite{jia2021reducing}, and simplify analysis of their algorithm. This is summarized in the following problem, which is a keystone in rounding schemes for general convex bodies (see Remark~\ref{rem:rounding-general-body}).

\begin{problem}\label{prob:rounding}
Let $\K\subset\Rd$ be a well-rounded convex body containing a ball of radius $1$ (i.e., $B_{1}(0)\subset\K$ and $\E_{\pi}[\norm{\cdot}^{2}]\lesssim d$ for the uniform distribution $\pi$ over $\mc K$). Can we find an algorithm that makes $\K$ $c$-isotropic for $c\approx1$ (i.e., $c^{-1}I_{d}\preceq\cov(\pi)\preceq cI_{d}$) using at most $\Otilde(d^{3})$ queries to the membership oracle\footnote{A membership oracle for $\K$ answers YES or NO to the query of the form ``$x\in \K$''?} of $\K$?
\end{problem}

We note that if we obtain the mean $\mu$ and covariance $\Sigma$ of $\pi$, then the transformed convex body $\Sigma^{-1/2}(\K-\mu)$ is isotropic (i.e., $\E X=0$ and $\E[X^{\otimes2}]=I_{d}$); this procedure is termed \emph{rounding}. 
As elaborated in the sequel, it takes roughly $d^{3}$ queries to get a single (approximately) uniform sample from $\K$, and one needs roughly $d$ independent samples in general to obtain an accurate estimator. In the case where the samples are required to be i.i.d., it is inevitable that we pay $d^{4}$ queries in total, which would answer Problem~\ref{prob:rounding} in the negative.

To circumvent this, an algorithm proposed by \cite{jia2021reducing} repeats `sampling $\to$ approximate covariance estimation $\to$ rounding',
which gradually isotropizes $\K$. We provide details of their argument in \S\ref{app:comparison}, where we compare against our approach. 

\subsubsection{Algorithm}
\paragraph{Uniform sampling by $\protect\ino$.}

\cite{kook2024and} proposes $\ino$, which is essentially the proximal sampler
for uniformly sampling from a convex body $\K$ containing a ball of radius $r$.
This iterates two steps for some suitable $h\asymp r^{2}/d^{2}$: (1) $y_{i+1}\sim\mc N(x_{i},hI_{d})$ and (2) $x_{i+1}\sim\mc N(y_{i+1},hI_{d})|_{\K}$, where the second step is implemented by rejection sampling with proposal $z\sim\mc N(y_{i+1},hI_{d})$.
They show that from an $\O(1)$-warm start, $\ino$ iterates $n=\otilde{qr^{-2}d^{2}\cpi(\pi)\polylog\frac{1}{\veps}}=\otilde{qr^{-2}d^{2}\norm{\Sigma}_{\op}\polylog\frac{1}{\veps}}$ times to find a sample $X_{n}$ with $\eu R_{q}(\law(X_{n})\mmid\pi)\leq\veps$, using $\Otilde(n)$ membership queries in expectation for success probability at least $1-\delta$. 
In particular, denoting by $P$ the Markov kernel of one iteration, they showed an exponential contraction of $\ino$ in $\chi^{2}$,
\begin{equation}
\chi^{2}(\mu P\mmid\pi)\le\frac{\chi^{2}(\mu\mmid\pi)}{(1+h/\cpi(\pi))^{2}}\,.\label{eq:chi-contraction}
\end{equation}
The spectral gap of this step is roughly lower bounded by $\Omega(h\cpi^{-1})=\Omega(r^{2}d^{-2}\cpi^{-1})$.
Therefore, defining a new Markov kernel by $\overline{P}:=P^{n}$ for $n=\O(r^{-2}d^{2}\cpi(\pi))$, we can ensure that the spectral gap of the new kernel $\overline{P}$ is at least $0.99$. We denote $\ino_{N}(\mu,\nu,h)$ for the Markov chain with kernel $P^{N}$, initial distribution $\nu$, and target distribution $\mu$.

\paragraph{High-level description.}

Let $\pi_{\K}$ denote the uniform distribution over a convex body $\mc K$, and $r:=\msf{inrad}(\mc K)$ denote the radius of the largest ball contained in $\mc K$. Recall that $\ino$ needs roughly $d^{2}\norm{\Sigma}_{\op}/r^{2}$ queries per sample from a warm start, which indicates that the mixing of the sampler suffers from the skewness of $\K$ (i.e., how far $\Sigma$ is from $I_d$).

 Algorithm~\ref{alg:iterative_rounding}, a modified version of
\cite[Algorithm 2]{jia2021reducing}, starts off by generating a warm
start via $\gc$~\cite{kook2024renyi} using $d^{3}$ membership
queries. Then it follows the sampling-and-rounding scheme as described earlier, initialized at this warm start. It draws $r^{2}$ many samples to get a rough estimate of the covariance up to a $d$-additive error (i.e., $\abs{\overline \Sigma - \Sigma} \preceq dI_d$), using $r^{2} \cdot d^{2}\norm{\Sigma}_{\op}/r^{2}=d^{2}\norm{\Sigma}_{\op}$
queries in total. Then, the subspace corresponding to eigenvalues
of the covariance $\Sigma$ that are smaller than $d$ is doubled,
which almost doubles $\msf{inrad}(\mc K)$ as well. Hence, even
though the algorithm requires nearly four times as many samples in
the next iteration, this increase is exactly balanced by the reduced
query complexity per sample. This is repeated until $r^{2}$ reaches
roughly $d$, which implies that the total number of iterations is
$\O(\log d)$, since $r$ almost doubles every iteration. Furthermore,
the largest eigenvalue of the covariance, $\norm{\Sigma}_{\op}$,
increases at most roughly by $d$ every iteration. Since the operator
norm of the initial covariance is at most $d$, the largest eigenvalue
remains $\Otilde(d)$ throughout. Thus, the total query complexity
per iteration is $\Otilde(d^{3})$; the query complexity of the algorithm is the same with only a logarithmic overhead.

In the algorithm, let $\pi_i$ denote $\pi_{\K_i}$, and $\Sigma_i$ be its covariance.
\begin{algorithm}[H]
\hspace*{\algorithmicindent} \textbf{Input:} convex body $\mc K\subset\Rd$
and $T_{1}\in\Rdd$ such that $\mc K_1:=T_1\mc K$ satisfies $\nicefrac{1}{4}\leq\msf{inrad}(\mc K_{1})$ and
$\ensuremath{\E_{\pi_{\mc K_{1}}}[\norm{\cdot}^{2}]\leq C^{2}d}$ for constant $C>0$. 

\hspace*{\algorithmicindent} \textbf{Output:} $\ensuremath{(\widehat{\mu},\widehat{\Sigma}^{-1/2})}.$

\begin{algorithmic}[1]

\STATE Run $\gc$ to obtain $X_{0}\in\mc K_{1}$
with $\eu R_{\infty}(\law(X_{0})\mmid\pi_{\mc K_{1}})\leq\log2$.\label{line:GC-warm}

\STATE Let $r_{1}=1/4$, $\pi_1 := \pi_{\K_1}$, and $i=1$.

\WHILE{$r_{i}^{2}\leq \tfrac{d}{2^{10}\log^{4}d}$} 

\STATE Set $k_{i}=10cr_{i}^{2}C^{2}\log^{6}C^{2}d$ for some universal
constant $c$. 

\STATE Draw $\{X_j\}_{j\in [k_i]} \gets \ino_{N_{i}}\bpar{\pi_{i},\delta_{X_{0}},\frac{r_{i}^{2}}{2^{10}d^{2}\log(Ccr_{i}d)}}$ with $N_{i}=\frac{2^{10}C^{2}d^{3}\log(Ccd)}{r_i^2}$. \label{line:INO-setting}

\STATE Compute $\widehat{\mu}_{i}=\frac{1}{k_{i}}\sum_{j=1}^{k_{i}}X_{j}$
and $\widehat{\Sigma}_{i}=\frac{1}{k_{i}}\sum_{j=1}^{k_{i}}(X_{j}-\widehat{\mu}_{i})^{\otimes2}$.
\label{line:mean-cov-estimate}

\STATE Compute $M_{i}=I_{d}+P_{i}$, where $P_{i}$ is the orthogonal projection
to the subspace spanned by eigenvectors of $\widehat{\Sigma}_{i}$
with eigenvalue at most $d$.

\STATE Set $T_{i+1}=M_{i}T_{i}$, $\K_{i+1}=M_{i}\K_i$, $X_{0}\gets M_{i}X_{0}$, $r_{i+1}=2r_{i}(1-\nicefrac{1}{\log d})$, and $i\gets i+1$.
\label{line:r-doubling}

\ENDWHILE

\STATE Draw $c'd\log^{6}d$ outputs from $\ino_{N}(\pi_{i},\delta_{X_{0}},\frac{1}{2^{20}d\log(Ccr_{i}d)})$
with $N_{i}=2^{20}C^{2}d^{2}\log^{5}(Ccd)$ (for some universal constant
$c'$), and use them to compute the mean $\widehat{\mu}$
and covariance $\widehat{\Sigma}$.\label{line:last-estimation}

\end{algorithmic}\caption{$\protect\msf{Isotropize}$ \label{alg:iterative_rounding}}
\end{algorithm}

\subsubsection{Comparison between approaches}\label{app:comparison}
\cite{jia2021reducing} used $\bw$ (implemented using $\sw$, followed by a rejection
step) started at an initial distribution $\mu$, whose query complexity
for obtaining a $\veps$-close sample (to $\pi$) in $\tv$ is $\O(Md^{2}\psi_{\text{KLS}}^{2}\log^{\O(1)}\frac{1}{\veps})$,
where $M=\sup_{\K}\frac{\mu}{\pi}$ is a warmness parameter, and $\psi_{\text{KLS}}$
is defined by $\psi_{\text{KLS}}^{-1}=\inf_{S\subset\Rd}\frac{\pi^{+}(S)}{\pi(S)\wedge\pi(S^{c})}$.
As for the approximate covariance estimation, they took the same approach as us; they draw a few samples to detect a subspace of small eigenvalues of the covariance matrix, followed by the upscaling of this `skewed' subspace.

When drawing samples used for a covariance estimation, they generate an $\O(1)$-warm sample and then initialize several parallel independent threads of $\bw$, with \emph{the warm sample} used as an initial distribution. 
Therefore, these $\veps$-accurate samples are independent \textbf{conditioned on} the warm sample, but still dependent overall. Then, they rely on a covariance estimation lemma assuming \emph{independence of these samples}, justifying their argument via a reference to \emph{$\mu$-independence} (also known as the $\alpha$\emph{-mixing} in the statistic literature). They mention that with probability $1-\delta$, this procedure, iterated  $\O(\log\frac{1}{\delta})$ more times, can succeed in estimating the covariance. We refer interested readers to their paper \cite[Computational Model]{jia2021reducing}, which references \cite[\S3.2]{lovasz2006simulated}.

We introduce two main changes to their approach -- (1) we take sequential samples instead of maintaining several independent threads and (2) we use $\ino$ instead of $\bw$. 
We believe this is more principled in that the query complexity $\O(Md^{2}\cpi(\pi)\log^{\O(1)}\frac{1}{\veps})$ of $\ino$ needed for convergence in $\chi^{2}$ has a direct connection to the Poincar\'e constant $\cpi(\pi)$ of the uniform distribution.
Thus, if we take $\O(d^{2}\cpi(\pi))$ iterations of $\ino$ as a single iteration
of a new Markov chain, then it immediately follows that the spectral
gap of the new chain can be made at least $0.99$ (or any constant as close to $1$ as desired). This allows us to use our result developed in \S\ref{sec:covar_est}, streamlining analysis for statistical estimation using dependent samples.
In addition to this, by incorporating the recent improvement of $\cpi(\pi)$~\cite{klartag2023logarithmic} (or equivalently $\psi_{\text{KLS}}^{2}$) from $d^{o(1)}\norm{\Sigma}_{\op}$
to $\norm{\Sigma}_{\op}\log d$, we bypass the anisotropic KLS bound
developed in \cite{jia2021reducing}. Combining all these changes
streamlines our analysis significantly.

\subsubsection{Analysis}
We analyze each algorithmic component under the event that all the previous subroutines succeed.
All the proofs for this section are deferred to \S\ref{app:rounding}.
\paragraph{(1) Guarantees of the sampler.}
First, we ensure that $\ino$ indeed satisfies Assumption~\ref{as:markov}, which is a consequence of the convergence rates established in earlier work~\cite{kook2024and}. 
\begin{lem}
\label{lem:INO-setting} $\ino_{N_{i}}(\pi_{i},\cdot,\frac{r_{i}^{2}}{2^{10}d^{2}\log(Ccr_{i}d)})$
with $N_{i}=2^{10}C^{2}d^{3}r_{i}^{-2}\log d$ has a spectral gap
of at least $0.99$ (or any desired constant approaching $1$).
\end{lem}
We present a version of Theorem~\ref{thm:add-cov-Poincare} tailored
to this application. Its proof essentially follows that of Theorem~\ref{thm:add-cov-Poincare}, using the spectral-gap condition of $\ino$ and the fact that the initial distribution $\law(X_0)$ is close to $\pi_1$ (see Line~\ref{line:GC-warm}).

\begin{lem}\label{lem:cov-estimation-app1} There exists a universal constant $c>0$ such that when $\ino_{N_i}$ successfully iterates without failure, each while-loop ensures that with probability at least $1 - \nicefrac{\fp}{d}$
\[
\frac{9}{10}\,\Sigma_{i}-\frac{d}{100}\,I_{d}\preceq\widehat{\Sigma}_{i}\preceq\frac{11}{10}\,\Sigma_{i}+\frac{d}{100}\,I_{d} \quad \text{ if } N\geq cd^{-1}\tr\Sigma_{i}\log^{6}C^{2}d\,.
\]
\end{lem}
Leveraging this, we can apply our Markovian covariance estimation machinery in the form of Corollary~\ref{cor:mul-cov-Poincare} to obtain a covariance concentration bound for Line~\ref{line:last-estimation}.

\paragraph{(2) Control over trace and operator norm.}
The following lemma establishes quantitative control over changes of the trace and operator norm of the covariance $\Sigma_{i}$ at each iteration.
We present a simpler proof of \cite[Lemma 3.1]{jia2021reducing}, bypassing the need for an anisotropic KLS bound.
\begin{lem}[Control over covariance] \label{lem:basic_property_isotropization}
The while-loop iterates at most $2\log d$ times. Also,
\begin{enumerate}
\item $\norm{\Sigma_{i}}_{1}=\tr\Sigma_{i}\leq10r_{i}^{2}C^{2}d$. 
\item $\|\Sigma_{i}\|_{\op}\le d(C^{2}+6i)$.
\end{enumerate}
\end{lem}

Lastly, we prove that the inner radius $\msf{inrad}(\K_{i})$ almost doubles every iteration, providing a rigorous proof of \cite[Lemma 3.2]{jia2021reducing}.
\begin{lem}[Control over $\msf{inrad}$]\label{lem:inner-rad-double}
Under $r_{i+1}=2(1-\nicefrac{1}{\log d})r_{i}$, each while-loop ensures $r_{i}\leq\msf{inrad}(\K_{i})$.
\end{lem}

Putting all of the aforementioned results together, we conclude that Algorithm~\ref{alg:iterative_rounding} returns $\widehat{\mu}$ and $\widehat{\Sigma}$ such that $\widehat{\Sigma}^{-1/2}(\K-\widehat{\mu})$ is nearly isotropic with high probability. 

\begin{lem}\label{lem:rounding-final}
Algorithm~\ref{alg:iterative_rounding} returns $(\widehat{\mu},\widehat{\Sigma})$ such that $\widehat{\Sigma}^{-1/2}(\mc K-\widehat{\mu})$ is $2$-isotropic with probability at least $1-\nicefrac{\fp}{\sqrt{d}}$, using $\Otilde(C^{4}d^{3})$ membership queries in expectation.
\end{lem}

\begin{rem}[General bodies]\label{rem:rounding-general-body}
Now we can design a rounding algorithm with roughly $d^{3.5}$ query complexity for a \emph{general convex body} containing a unit ball, which combines Algorithm~\ref{alg:iterative_rounding} with the annealing algorithm in~\cite{jia2021reducing}. Given a well-rounded uniform distribution over $T(\K\cap B_r(0))$ for some $r\geq 1$ and affine map $T:\R^{d\times d}\to \Rd$, Algorithm~\ref{alg:iterative_rounding} will find a new affine map $T'$ such that $T'(\K\cap B_r(0))$ is nearly isotropic with high probability. The annealing algorithm then moves to the uniform distribution over $T'(\K \cap B_{r(1+d^{-1/2})}(0))$, which can be shown to be well-rounded by \cite[Lemma 3.4]{jia2024reducingisotropyvolumekls}.
This annealing algorithm begins with a ball of radius $r=1$ (which is surely well-rounded) and repeats the procedure above until the radius reaches the diameter $D$ of $\K$. The total number of iterations for this annealing procedure is $d^{1/2}\log D$. Multiplying by the complexity of Algorithm~\ref{alg:iterative_rounding}, we find that the total complexity of the entire rounding algorithm is $\Otilde(d^{3.5}\log D)$.
\end{rem}

\subsection{Covariance estimation in unconstrained sampling}

In a similar vein, we state and prove a guarantee for covariance estimation when the target is an unconstrained distribution satisfying a Poincar\'e inequality. Briefly, we note that the analysis of $\ps$ (for which $\ino$ is the constrained equivalent) elegantly relates the mixing of the sampler in various divergences to the isoperimetric constants of the target $\pi$; in particular, there is a fundamental relationship between the Poincar\'e constant of $\pi$ and the mixing rate in $\chi^2$.

\paragraph{The $\ps$ for unconstrained distributions.}
Below, we recall the $\ps$ when sampling from an unconstrained distribution $\pi \propto \exp(-V)$.
$\ps$ iterates, for some step size $h>0$: (Forward) $y_{i+1}\sim\mc N(x_{i},hI_{d})$ and (Backward) $x_{i+1} \sim Q_h(\cdot\,|\,y_{i+1})$, where $Q_h$ has density
\[
    Q_h(\cdot\, |\,y_{i+1}) \propto \exp\bigl(-V(\cdot) - \frac{1}{2h}\, \norm{\cdot - y_{i+1}}^2 \bigr)\,.
\]
Similar to $\ino$, this procedure can be seen as Gibbs sampling. However, unlike $\ino$, the implementation of the reverse step is not straightforward and requires some additional effort. 

For their state-of-the-art result, the full methodology of~\cite{altschuler2024faster} uses a composite algorithm to (approximately) implement the backwards step. The composite algorithm consists of (1) the Metropolis adjusted Langevin algorithm (MALA), a high-accuracy sampler which performs well when given a warm start, and (2) the low-accuracy underdamped Langevin Monte Carlo (ULMC) sampler in order to generate that warm start. We do not explicitly give the construction of this proposal in this work, and invite the reader to peruse~\cite{altschuler2024faster} for additional details.

An important detail is that this only generates \emph{approximate} samples from the distribution $Q_h$ of the backwards step, with some chosen error tolerance $\varrho$ so that the final output is close to $Q_h$. We refer to the composition of the exact forward kernel with the inexact reverse kernel as $\hat P_{h, \varrho}$. The overall methodology is summarized in Algorithm~\ref{alg:unconstrained}.
\begin{algorithm}
\hspace*{\algorithmicindent} \textbf{Input:} $\pi \propto \exp(-V) \in \mc P_2(\R^{d})$ such that $V$ is $\beta$-smooth and $C_{\PI}(\pi) < \infty$, target error $\varepsilon > 0$.

\hspace*{\algorithmicindent} \textbf{Output:} $\ensuremath{(\hat \mu, \hat\Sigma)}$.

\begin{algorithmic}[1]

\STATE Let $K = \frac{c_K d \phi}{\varepsilon^2} \log^2 \frac{d\phi}{\varepsilon} \log^4 d$, $h = \frac{1}{2\beta}$, $n_0 = c_{n_0} \kappa (d \vee \beta) \log(\kappa d)$, $n = c_n \kappa$, $\varrho = \frac{c_{\varrho} (Kn + n_0)}{d}$, where $c_K, c_{n_0}, c_n, c_{\varrho}$ are all positive universal constants, $\Sigma = \cov(\pi), \phi \deq \frac{\tr \Sigma + C_{\PI}(\pi)}{\tr \Sigma}$, $\kappa \deq \beta C_{\PI}(\pi)$.

\STATE Obtain $X_0 \sim \mc N(0, \beta^{-1}I_d)$.

\FOR{$j \in [K]$} 

\STATE Draw $X_j \sim \delta_{X_{j-1}}\hat P_{h, \varrho}^n$.

\ENDFOR

\STATE Compute the sample mean and covariance $\hat \mu \gets \frac{1}{K} \sum_{j=1}^K X_j$, $\hat \Sigma \gets \frac{1}{K} \sum_{j=1}^K (X_j - \hat \mu)^{\otimes 2}$.

\end{algorithmic}\caption{Covariance estimation in the unconstrained setting \label{alg:unconstrained}}
\end{algorithm}

Finally, we note that the approximate kernels used at different steps are allowed to be different, and the user is free to choose the error tolerance at each step so that their final guarantee is suitably strong. For notational simplicity, we assume that the error tolerances do not differ from step to step.

\paragraph{Results.}
In this setting, we will primarily operate with the following standard assumptions.
\begin{assumption}\label{as:smoothness}
    The target distribution $\pi \propto \exp(-V)$ satisfies \eqref{eq:poincare} and $V$ is $\beta$-smooth; 
    \[
        \norm{\nabla V(x) - \nabla V(y)} \leq \beta \norm{x-y} \text{ for all } x, y \in \Rd\,.
    \]
\end{assumption}
For instance, when $\pi \propto \exp(-\sqrt{1+\norm{x}^2})$, we have by~\cite{bobkov2003spectral} that $C_{\PI}(\pi) = \mc O(d)$. The perturbation principle of Holley and Stroock~\cite{holley1986logarithmic} also states that for $\tilde \pi \propto \exp(-V + f)$, where $\norm{f}_\infty < \infty$ and $V$ is everywhere finite, a Poincar\'e inequality continues to hold for $\tilde \pi$ with constant depending on $\norm{f}_\infty$. In the sequel, we will denote the ``condition number'' by $\kappa \deq C_{\PI}(\pi) \beta$, which recovers the standard condition number $\beta/\alpha$ when $\pi$ is $\alpha$-strongly log-concave. 
\begin{lem}[{\cite{chen2022improved}, Theorem 4}]\label{lem:unconstrained-contraction}
    Suppose $\pi$ satisfies Assumption~\ref{as:smoothness}. Let $P_h$ be the Markov kernel corresponding to $\ps$ with stationary measure $\pi \in \mc P(\Rd)$ and step size $h$. Then, $P_h$ satisfies, for any initial measure $\mu \in \mc P(\Rd)$
    \[
        \chi^2(\mu P_h\mmid \pi) \leq \frac{\chi^2(\mu \mmid \pi)}{(1+h/C_{\PI}(\pi))^{2}}\,.
    \]
\end{lem}
In the sequel, we will also impose the following standard oracle model.
\begin{assumption}\label{as:prox-oracle}
    Assume that we have access to oracles for $V$, $\nabla V$, and also that we have access to the proximal oracle with step size $h = \nicefrac{1}{2\beta}$ for $V$, which given a point $y \in \R^d$ returns $\argmin_{x \in \R^d} \{V(x)  +\frac{1}{2h} \norm{x-y}^2\}$.
\end{assumption}
\begin{rem}
    The proximal oracle is not strictly necessary, and can be removed by following the same techniques as~\cite{altschuler2024faster}, at the possible cost of additional polylogarithms in the bound.
\end{rem}

In this section, when we speak about expected query complexity, we are referring to the sum total of queries either to oracles for $V, \nabla V$ or to the proximal oracle given above. Hereafter we will also suppress the subscripts for $h, \varrho$ in the kernel, as these remain fixed throughout the algorithm. The query complexity guarantees from~\cite{altschuler2024faster} are summarized below.
\begin{lem}[{Adapted from~\cite[Theorem D.1]{altschuler2024faster}}]\label{lem:alg-iid}
    In the setting of Lemma~\ref{lem:unconstrained-contraction}, under Assumptions~\ref{as:smoothness} and~\ref{as:prox-oracle}, there exists an algorithm which, given any initial point $x \in \R^d$, returns a point $z \sim \delta_x \hat P$, with $\KL(\delta_x \hat P \mmid \delta_x P) \leq \varrho$ for any $\varrho \in (0, \nicefrac{1}{2}]$. Its expected query complexity is $N = \Otilde\bigl(d^{1/2} \log^3 \frac{1}{\varrho}\bigr)$.
\end{lem}
\begin{rem}
Due to its reliance on an inexact implementation of the proximal sampler, we do not know if the algorithm can be written in the form of a kernel possessing a spectral gap. As a result, we cannot apply Theorem~\ref{thm:mul-cov-Poincare} directly to this algorithm. Nonetheless, the error guarantees are sufficient to retain the concentration results developed in this work.
\end{rem}
We can now convert this error into a total variation bound on the entire chain, as below.
\begin{lem}\label{lem:tv-unconstrained-chain}
    In the setting of Lemma~\ref{lem:alg-iid}, set $\nu = \mu_0 P^{n_0}$ and $\hat \nu = \mu_0 \hat P^{n_0}$ for any initialization $\mu_0$ and $n_0 \in \mathbb N$, where $\varrho \asymp \nicefrac{\delta}{(K\kappa +n_0)}$ is chosen in Lemma~\ref{lem:alg-iid}, and $K \in \mathbb N$. For $n = \O(\kappa)$, let $\nu_{1:K}, \hat \nu_{1:K}$ be the joint laws of $K$ iterates drawn from Markov chains with initial distributions $\nu, \hat \nu$ and kernels $P^n, \hat P^n$ respectively. Then, for a given $\delta \in (0, \nicefrac{1}{2}]$, we can guarantee that $\norm{\nu_{1:K}-\hat \nu_{1:K}}_{\msf{TV}} \leq \delta$. The expected query complexity of implementing the chain corresponding to $\hat \nu_{1:K}$ is
    \[
        N = \Otilde\bigl(n_0 d^{1/2} \log^3 \frac{K}{\delta}\bigr) + \Otilde\bigl(K\kappa d^{1/2} \log^3 \frac{n_0}{\delta} \bigr)\,.
    \]
\end{lem}

We will take $\delta \asymp \nicefrac{1}{d}$. Lemma~\ref{lem:unconstrained-contraction} suggests that we set $n_0 = \kappa \log \chi^2(\mu_0 \mmid \pi)$, and so we need $\Otilde(\kappa d^{1/2} \log \chi^2(\mu_0 \mmid \pi) \polylog K)$ queries to obtain a sample whose law is $\O(1)$-close to $\pi$ in $\chi^2$-divergence. Under standard assumptions (Lemma~\ref{lem:initialization}), $\log \chi^2(\mu_0 \mmid \pi) \asymp \Otilde(d \vee \beta)$, so this complexity would be $\Otilde(\kappa d^{1/2}(d \vee \beta))$. For the requisite $K$ \emph{independent} samples, we repeat the whole procedure from scratch, expecting $\Otilde(\kappa d^{1/2}(d \vee \beta) \polylog (K+n_0))$ queries per sample to construct an accurate covariance estimate.

In contrast, after using $\Otilde(\kappa d^{1/2}(d \vee \beta) \polylog K)$ queries to obtain a warm sample, an estimator based on Markovian iterates needs $\Otilde(\kappa d^{1/2} \polylog (K + n_0))$ for each further sample.
\begin{thm}[Covariance estimation in unconstrained sampling]\label{thm:unconstrained-covar}
    In addition to Assumptions~\ref{as:smoothness} and~\ref{as:prox-oracle}, assume that $V(0) - \min V \lesssim d$, $\E_\pi \norm{\cdot} = \poly(\kappa, d)$, and $\nabla V(0) = 0$. 
    We then have with probability $1-\nicefrac{\fp}{d}$ that the matrix obtained by Algorithm~\ref{alg:unconstrained} satisfies
    \[
        \norm{\hat \Sigma - \Sigma}_{\op} \leq \veps\norm{\Sigma}_{\op}\,.
    \]
    In particular, the total query complexity is bounded in expectation as
    \[
        N = \Otilde\bigl(\max \bigl\{\kappa d^{1/2} (d \vee \beta) \log^3 \frac{\phi}{\varepsilon}, \frac{\kappa d^{3/2} \phi}{\varepsilon^2} \bigr\}\bigr)\,.
    \]
\end{thm}
\begin{rem}\label{rmk:unconstrained-iid}
It can be seen that using $d$ i.i.d. samples would require $\Otilde(\frac{\kappa d^{3/2}(d \vee \beta) \phi}{\varepsilon^2})$ queries in expectation. This is always worse than the rate above by a factor of at least $d$. Note that Theorem~\ref{thm:unconstrained-covar} also implies a concentration result for the mean, and a similar proof will yield guarantees resembling that of Theorem~\ref{thm:add-cov-Poincare}, albeit with additional terms in the query complexity.
\end{rem}

\section{Details for the main results}\label{app:proofs-main}

\subsection{Matrix Bernstein inequality under spectral gap of a Markov chain}\label{app:mtx-Bern}

\begin{proof}[{Proof of Corollary~\ref{cor:vec-bern-spectral}}]
We deduce this from Theorem~\ref{thm:dependent-bernstein}. Consider the matrix 
\[
F(X_{i}):=\left[\begin{array}{cc}
0 & X_{i}^{\T}\\
X_{i} & 0_{d\times d}
\end{array}\right]\,.
\]
Clearly, $\E[F(X_{j})]=0$. Next, its operator norm is bounded as
follows: for $(v,w)\in\R\times\R^{d}$ with $\norm{(v,w)}=1$, 
\[
\norm{F(X_{j})}=2\sup_{\norm{(v,w)}=1}v\,X_{j}^{\T}w\leq2\sup|v|\norm w\norm{X_{j}}\leq\norm{X_{j}}\leq\mc M\,,
\]
and the supremum is achieved by $v=1/\sqrt{2}$ and $w=X_{j}/(\sqrt{2}\norm{X_{j}})$.
Also, observe that 
\[
\E[F(X_{j})^{2}]=\left[\begin{array}{cc}
\E[\norm{X_{j}}^{2}] & 0\\
0 & \E[X_{j}X_{j}^{\T}]
\end{array}\right]\preceq\E[\norm{X_{j}}^{2}]\cdot I_{d+1}\,,
\]
where the inequality follows from $\E[X_{j}X_{j}^{\T}]\preceq\E[\norm{X_{j}}^{2}]\cdot I_{d}$.
Therefore, $\norm{\E[F(X_{j})^{2}}\leq\E[\norm{X_{j}}^{2}]\leq\mc V$.

Noting that the operator norm of 
\[
\sum_{j=1}^{n}F(X_{j})=\left[\begin{array}{cc}
0 & \sum X_{j}^{\T}\\
\sum X_{j} & 0
\end{array}\right]
\]
is $\norm{\sum X_{j}}$, and using Theorem~\ref{thm:dependent-bernstein},
\[
\P\Bpar{\Bnorm{\sum_{j=1}^{n}X_{j}}\geq t}=\P\Bpar{\Bnorm{\sum_{j=1}^{n}F(X_{j})}\geq t}\leq(d+1)^{2-\pi/4}\exp\Bpar{\frac{-t^{2}/(32/\pi^{2})}{\alpha(\lda)\,\sigma^{2}+\beta(\lda)\,\mc Mt}}\,,
\]
which completes the proof. 
\end{proof}

The following lemma allows us to perform most of our calculations
at stationarity, paying only a small overhead. 
\begin{lem}[Change of measure]\label{lem:change-measure} 
Suppose a Markov chain with kernel $P$ starts from $\nu$ instead of the stationary measure $\pi$.
Then, the probability of an event has a multiplicative factor of $1+\chi^{2}(\nu\mmid\pi)$.
\end{lem}

\begin{proof}
Let $A$ be the bad event we want to bound. Let us denote $\pi_{1:n}:=\law(s_{1},\dotsc,s_{n})$ for $s_{1}\sim\pi$ and $\nu_{1:n}:=\law(s_{1},\dotsc,s_{n})$ for $s_{1}\sim\nu$, where $s_{i+1} \sim P(\cdot | s_i)$. Then, 
\[
\nu_{1:n}(A)=\int\ind_{A}\,\D\nu_{1:n}=\int\ind_{A}\frac{\D\nu_{1:n}}{\D\pi_{1:n}}\,\D\pi_{1:n}\leq\bpar{1+\chi^{2}(\nu_{1:n}\mmid\pi_{1:n})}\cdot\pi_{1:n}(A)\,.
\]
Note that 
\begin{align*}
1+\chi^{2}(\nu_{1:n}\mmid\pi_{1:n}) & =\int\Bpar{\frac{\nu_{1:n}}{\pi_{1:n}}}^{2}\D\pi_{1:n}=\int\int\Bpar{\frac{\nu_{1:n-1}}{\pi_{1:n-1}}}^{2}\Bpar{\frac{\nu_{n|-n}}{\pi_{n|-n}}}^{2}\,\D\pi_{n|-n}\D\pi_{1:n-1}\\
 & \underset{(i)}{=}\int\int\Bpar{\frac{\nu_{1:n-1}}{\pi_{1:n-1}}}^{2}\Bpar{\frac{\nu_{n|n-1}}{\pi_{n|n-1}}}^{2}\,\D\pi_{n|n-1}\D\pi_{1:n-1}\\
 &\underset{(ii)}{=}\int\Bpar{\frac{\nu_{1:n-1}}{\pi_{1:n-1}}}^{2}\D\pi_{1:n-1}\\
 &\underset{(iii)}{=}\int\Bpar{\frac{\nu_{1}}{\pi_{1}}}^{2}\D\pi_{1} =1+\chi^{2}(\nu_{1}\mmid\pi_{1})\,,
\end{align*}
where $(i)$ follows from $P$ being a Markov chain, $(ii)$ follows
from $\nu_{n|n-1}=\pi_{n|n-1}$, and $(iii)$ follows from induction.
\end{proof}

\subsection{Moments under Poincar\'e}

We can establish concentration of the target measure around its mean.
\begin{lem}\label{lem:tail-Poincare} 
Suppose that $\pi \in \mc P(\Rd)$ satisfies \eqref{eq:poincare}. Then, letting $\mu \in \R^d$ denote the mean of $\pi$, for all $t\geq1$ we have
\[
\P(\norm{X-\mu}\geq t\sqrt{\tr\Sigma})\leq\exp\Bpar{-(t-1)\,\bpar{\frac{\tr\Sigma}{\cpi(\pi)}}^{1/2}}\,.
\]
\end{lem}

\begin{proof}
We note that $f(x)=\norm{x-\mu}$ is $1$-Lipschitz. Hence, it follows from the Lipschitz concentration property in \eqref{eq:poincare-Lip-con} that for all $t\geq0$,
\[
\P(\norm{X-\mu}\geq t+\E\norm{X-\mu})\leq3\exp\bpar{-\frac{t}{\sqrt{C_{\PI}(\pi)}}}\,.
\]
Since $\E\norm{X-\mu}\leq(\E[\norm{X-\mu}^{2})^{1/2}=\sqrt{\tr\Sigma}$,
we can deduce that 
\[
\P(\norm{X-\mu}\geq t+\sqrt{\tr\Sigma})\leq3\exp\bpar{-\frac{t}{\sqrt{C_{\PI}(\pi)}}}\,,
\]
which completes the proof.
\end{proof}
One can also bound the fourth-moment under the Poincar\'e inequality:
\begin{lem}\label{lem:4th-Poincare} 
Suppose that $\pi\in \mc P(\Rd)$ satisfies \eqref{eq:poincare}. Then,
\[
\E[\norm{X-\mu}^{4}]\leq(4\cpi(\pi)+\tr\Sigma)\tr\Sigma\,.
\]
\end{lem}
\begin{proof}
Let us assume that $\mu=0$ by translation. Taking $f(x):=\norm x^{2}$
in \eqref{eq:poincare}, we have from $\nabla f(x)=2x$ and $\E[\norm X^{2}]=\tr\Sigma$
that 
\[
\E[\norm X^{4}]\leq4\cpi(\pi)\,\E[\norm X^{2}]+(\E[\norm X^{2}])^{2}=\tr\Sigma\,(4\cpi(\pi)+\tr\Sigma)\,.\qedhere
\]
\end{proof}

\subsection{Proof details}
\subsubsection{Mean estimation}\label{app:mean-proof}

The proof of the mean estimation result closely mirrors that of the succeeding covariance estimation bound.

\begin{lem}\label{lem:mean-estimation} 
Under Assumption~\ref{as:markov}, if $\lambda\geq0.99$, then with probability at least $1 - \nicefrac{\fp}{d}$ we have that for any $\veps\in(0,d)$,
\[
\norm{\overline{X}-\mu}^{2}\leq\varepsilon\,,
\]
so long as $N\asymp\frac{\tr\Sigma+\cpi(\pi)}{\veps}\log^{2}\frac{d\max(d,\tr\Sigma)}{\varepsilon}\log^{2}d$. 
\end{lem}

\begin{proof}
By translation, we may assume that $\pi$ is centered (i.e., $\mu=\E_{\pi}X=0$).
Let $B_{t}:=\{x\in\Rd:\norm x\leq t\sqrt{\tr\Sigma}\}$, where $t>0$
are parameters to be determined, and we drop the subscript $t$ where
there is no confusion. Then, we have the following decomposition:
\[
\frac{1}{N}\sum X_{i}=\underbrace{\frac{1}{N}\sum(X_{i}\ind_{B}-\E[X_{i}\ind_{B}])}_{\eqqcolon\msf B_1}+\underbrace{\frac{1}{N}\sum X_{i}\ind_{B^{c}}}_{\eqqcolon\msf B_2}-\underbrace{\E[X\ind_{B^{c}}]}_{\eqqcolon\msf B_3}\,.
\]
\paragraph{Term $\protect\msf B_1$:}

As for this term, we essentially work with the truncated variables
$X_{i}\ind_{B}$. Let $F_{i}(X_{i})=\frac{1}{N}(X_{i}\ind_{B}-\E[X_{i}\ind_{B}])$
for some parameter $\lambda>0$. Clearly, $\E[F_{i}(X_{i})]=0$, and
\[
\norm{F_{i}(X_{i})}\leq\frac{1}{N}\norm{X_{i}\ind_{B}}+\frac{1}{N}\E\norm{X_{i}\ind_{B}}\leq\frac{2}{N}t\sqrt{\tr\Sigma}\,.
\]
For the variance, using the identity $(a+b)^{2}\leq2(a^{2}+b^{2})$
for $a,b\in\R$,
\begin{align*}
\E[\norm{F_{i}(X_{i})}^{2}] & \leq\frac{2}{N^{2}}\E[\norm{X_{i}\ind_{B}}^{2}]+\frac{2}{N^{2}}\norm{\E[X_{i}\ind_{B}]}^{2}\\
 & \leq\frac{4}{N^{2}}\E[\norm{X_{i}\ind_{B}}^{2}]\leq\frac{4t^{2}\tr\Sigma}{N^{2}}\,.
\end{align*}
Hence, we can substitute $\mc M=\frac{2t\sqrt{\tr\Sigma}}{N}$, $\mc V_{i}=\frac{4t^{2}\tr\Sigma}{N^{2}}$,
and $\sigma^{2}=\frac{4t^{2}\tr\Sigma}{N}$ into Corollary~\ref{cor:vec-bern-spectral},
finding that for $\alpha(\lda)=\frac{2-\lda}{\lambda}$ and $\beta(\lda)=\frac{8/\pi}{\lambda}$,
\[
\P\Bpar{\Bnorm{\frac{1}{N}\sum_{i=1}^{N}X_{i}\ind_{B}-\E[X_{i}\ind_{B}]}\geq\ell}\leq(d+1)^{2-\pi/4}\exp\Bpar{-\frac{\pi^{2}\ell^{2}/32}{\alpha(\lambda)\sigma^{2}+\beta(\lambda)\mc M\ell}}\,.
\]
Since $\alpha(\lambda)$ and $\beta(\lambda)$ are absolute constants, we have that with probability $1-\nicefrac{\fp}{d}$, 
\[
\Bnorm{\frac{1}{N}\sum_{i=1}^{N}X_{i}\ind_{B}-\E[X_{i}\ind_{B}]}\lesssim(\mc M+\sigma)\log d\lesssim\frac{t\sqrt{\tr\Sigma}\log d}{\sqrt{N}}\,.
\]

\paragraph{Terms $\protect\msf B_2$ and $\protect\msf B_3$:}

As for $\msf B_2$, we use Markov's inequality, 
\[
\P\Bpar{\Bnorm{\frac{1}{N}\sum X_{i}\ind_{B^{c}}}\geq s\,\E\Bbrack{\Bnorm{\frac{1}{N}\sum X_{i}\ind_{B^{c}}}}}\leq\frac{1}{s}\qquad\text{for any }s>0\,.
\]
Using Lemma~\ref{lem:tail-Poincare} for $t\geq2$ in the second last line below,
\begin{align}
\E\Bbrack{\Bnorm{\frac{1}{N}\sum X_{i}\ind_{B^{c}}}} & \leq\E\Bbrack{\frac{1}{N}\sum_{i=1}^{N}\norm{X_{i}}\ind_{B^{c}}}=\E_{\pi}[\norm X\ind_{B^{c}}]\nonumber \leq\sqrt{\E[\norm X^{2}]}\sqrt{\P(\norm X\geq t\sqrt{\tr\Sigma})}\nonumber \\
 & \leq3\sqrt{\tr\Sigma}\cdot\exp\Bpar{-\frac{(t-1)\sqrt{\tr\Sigma}}{\sqrt{\cpi(\pi)}}}\nonumber \leq3\sqrt{\tr\Sigma}\exp\Bpar{-\frac{t}{2}\,\bpar{\frac{\tr\Sigma}{\cpi(\pi)}}^{1/2}}\,.\label{eq:poin-bound-BC}
\end{align}
This bounds the norm of Term $\msf B_3$ as well by
the same argument, since
\[
\norm{\E[X\ind_{B^{c}}]}\leq\E[\norm X\ind_{B^{c}}]\leq3\sqrt{\tr\Sigma}\exp\Bpar{-\frac{t}{2}\,\bpar{\frac{\tr\Sigma}{\cpi(\pi)}}^{1/2}}\,.
\]

Combining all bounds on the three terms, it follows that with probability $1-\nicefrac{\fp}{s}-\nicefrac{\fp}{d}$, 
\[
\Bnorm{\frac{1}{N}\sum_{i=1}^{N}X_{i}}\lesssim\frac{t\sqrt{\tr\Sigma}\log d}{\sqrt{N}}+(s+1)\sqrt{\tr\Sigma}\exp\Bpar{-\frac{t}{2}\,\bpar{\frac{\tr\Sigma}{\cpi(\pi)}}^{1/2}}\,.
\]
Taking $s=8d$ and $t\asymp(\frac{\tr\Sigma+\cpi(\pi)}{\tr\Sigma})^{1/2}\log\frac{d\max(d,\tr\Sigma)}{\varepsilon}$,
we have 
\[
\Bnorm{\frac{1}{N}\sum_{i=1}^{N}X_{i}}\lesssim\varepsilon^{1/2}+\bpar{\frac{\tr\Sigma+\cpi(\pi)}{N}}^{1/2}\log d\log\frac{d^{2}\tr\Sigma}{\varepsilon}\,.
\]
Therefore, it suffices to take $N\asymp\frac{\tr\Sigma+\cpi(\pi)}{\veps}\log^{2}d\log^{2}\frac{d\max(d,\tr\Sigma)}{\varepsilon}$
and then revert the translation for centering.
\end{proof}

For the multiplicative version, we can just work with the relative error $\veps \norm{\Sigma}_{\op}$ in place of $\veps$.
\begin{cor}\label{cor:mean-multi}
    Under Assumption~\ref{as:markov}, if $\lambda\geq0.99$, then with probability at least $1-\nicefrac{\fp}{d}$ we have that for any $\veps\in(0,d)$,
\[
\norm{\overline{X}-\mu}^{2}\leq\varepsilon\norm{\Sigma}_{\op}\,,
\]
so long as $N\asymp\frac{d}{\veps}(1\vee \frac{\cpi(\pi)}{\tr \Sigma}\log^2\frac{d}{\veps})\log^{2}d $.
\end{cor}

\subsubsection{Covariance estimation}\label{app:cov-proof}

\paragraph{Additive version.}
We extend \cite[Lemma A.2]{jia2021reducing}, proven for independent samples distributed according to any log-concave distribution.
The following lemma contains the bound for the covariance part of our problem.

\begin{lem}\label{lem:cov-estimation} Under Assumption~\ref{as:markov}, if $\lambda\geq0.99$, then with probability at least $1-\nicefrac{\fp}{d}$
we have 
\[
\Bigl|\frac{1}{N}\sum (X_{i}-\mu)^{\otimes2}-\Sigma\Bigr|\preceq\veps\Sigma+\delta I
\]
so long as $N\asymp\frac{\tr\Sigma+\cpi(\pi)}{\veps\delta}\log^{2}\frac{d\,\max(d,\cpi(\pi)+\tr\Sigma)}{\delta}\log^{2}d$. 
\end{lem}

In the analysis, we can make $\pi$ centered (i.e., $\mu=0$) by translation and work
with a truncated region $B_{\rho,t}:=\{x\in\Rd:\norm x_{(\rho\Sigma+I)^{-1}}\leq t\sqrt{\tr\Sigma}\}$ with $t$ and $\rho$ to be determined; we again suppress the subscripts. Recall from \S\ref{sec:covar_est} that we have the following decomposition:
\[
\frac{1}{N}\sum X_{i}^{\otimes2}-\Sigma=\underbrace{\frac{1}{N}\sum(X_{i}^{\otimes2}\ind_{B}-\E[X_{i}^{\otimes2}\ind_{B}])}_{\eqqcolon\msf A_1}+\underbrace{\frac{1}{N}\sum X_{i}^{\otimes2}\ind_{B^{c}}}_{\eqqcolon\msf A_2}-\underbrace{\E[X^{\otimes2}\ind_{B^{c}}]}_{\eqqcolon\msf A_3}\,.
\]
We now bound each term separately. 
\begin{lem}[Term $\msf A_1$] \label{lem:covar_termA} With probability at least $1-\nicefrac{\fp}{d}$, for $t\geq2$,
\[
\Bigl|\frac{1}{N}\sum(X_{i}^{\otimes2}\ind_{B}-\E[X_{i}^{\otimes2}\ind_{B}])\Bigr|\preceq\veps\Sigma+\O\bpar{\frac{t^{2}\tr\Sigma \log^{2}d}{\veps N}}I\,.
\]
\end{lem}

\begin{proof}
We use Theorem~\ref{thm:dependent-bernstein},
setting $F_{i}(X_{i})=\frac{1}{N}(\rho\Sigma+I)^{-1/2}(X_{i}^{\otimes2}\ind_{B}-\E[X_{i}^{\otimes2}\ind_{B}])(\rho\Sigma+I)^{-1/2}$ for $i\in[N]$. 
We will only use the guarantees of Theorem~\ref{thm:dependent-bernstein} in the form of bounds on the operator norm, since this will suffice for our purposes. Denoting $X_{i}^{\otimes2}/(\rho\Sigma+I):=(\rho\Sigma+I)^{-1/2}X_{i}^{\otimes2}(\rho\Sigma+I)^{-1/2}$,
we have
\begin{align*}
\norm{F_{i}(X_{i})} & \leq\frac{\ind_{B}}{N}\Bnorm{\frac{X_{i}^{\otimes2}}{\rho\Sigma+I}}+\frac{1}{N}\E\Bbrack{\Bnorm{\frac{X_{i}^{\otimes2}}{\rho\Sigma+I}}\,\ind_{B}}\\
 & =\frac{1}{N}\norm{X_{i}}_{(\rho\Sigma+I)^{-1}}^{2}\ind_{B}+\frac{1}{N}\E[\norm{X_{i}}_{(\rho\Sigma+I)^{-1}}^{2}\ind_{B}]\\
 & \leq\frac{2}{N}t^{2}\tr\Sigma\,,
\end{align*}
where the last line follows from the definition of $B$. Hence, we
set $\mc M=2t^{2}\tr\Sigma/N$ in the theorem. As for the variance,
\begin{align*}
\E[F_{i}(X_{i})^{2}] & \underset{(i)}{\preceq}\frac{2}{N^{2}}\E\bbrack{\bpar{\frac{X_{i}^{\otimes2}}{\rho\Sigma+I}}^{2}\ind_{B}}+\frac{2}{N^{2}}\bpar{\E\bbrack{\frac{X_{i}^{\otimes2}}{\rho\Sigma+I}\,\ind_{B}}}^{2}\\
 & \underset{(ii)}{\preceq}\frac{2}{N^{2}}\E\bbrack{\norm{X_{i}}_{(\rho\Sigma+I)^{-1}}^{2}\frac{X_{i}^{\otimes2}}{\rho\Sigma+I}\,\ind_{B}}+\frac{2}{N^{2}}\bnorm{\E\bbrack{\frac{X_{i}^{\otimes2}}{\rho\Sigma+I}\,\ind_{B}}}\,\E\bbrack{\frac{X_{i}^{\otimes2}}{\rho\Sigma+I}\,\ind_{B}}\\
 & \underset{(iii)}{\preceq}\frac{2t^{2}\tr\Sigma}{N^{2}}\,\frac{\Sigma}{\rho\Sigma+I}+\frac{2t^{2}\tr\Sigma}{N^{2}}\,\frac{\Sigma}{\rho\Sigma+I}=\frac{2\mc M}{N}\,\frac{\Sigma}{\rho\Sigma+I}\,,
\end{align*}
where in $(i)$ we used $(A-B)^{2}\preceq2(A^{2}+B^{2})$ for matrices
$A$ and $B$, in $(ii)$ used $A^{2}\preceq\norm AA$ for a positive semidefinite matrix $A$,
and $(iii)$ follows from $\E[X_{i}^{\otimes2}]=\Sigma$ and the definition
of $B$. Hence,
\[
\norm{\E[F_{i}(X_{i})^{2}]}\leq\frac{2\mc M}{N}\,\Bnorm{\frac{\Sigma}{\rho\Sigma+I}}\,,
\]
so we can set $\mc V_i=2\mc M\norm{\Sigma/(\rho\Sigma+I)}/N$ and
$\sigma^{2}=2\mc M\norm{\Sigma/(\rho\Sigma+I)}$.

Using Theorem~\ref{thm:dependent-bernstein} with $\alpha(\lda),\beta(\lda)=\O(1)$,
we have 
\begin{align*}
&\P\Bpar{\Bnorm{\frac{1}{N}(\rho\Sigma+I)^{-1/2}\sum(X_{i}^{\otimes2}\ind_{B}-\E[X_{i}^{\otimes2}\ind_{B}])(\rho\Sigma+I)^{-1/2}}\geq\ell}\\
&\qquad\qquad\qquad\qquad\leq d^{2-\pi/4}\exp\Bpar{-\frac{\ell^{2}/(32/\pi^{2})}{\alpha(\lda)\,\sigma^{2}+\beta(\lda)\,\mc M\ell}}\,.
\end{align*}
Hence, with probability at least $1-\nicefrac{\fp}{d}$, using Young's inequality
with any $c>0$,
\begin{align*}
\norm{(\rho\Sigma+I)^{-1/2}\msf A_1(\rho\Sigma+I)^{-1/2}} & \lesssim(\mc M+\sigma)\log d\leq\Bpar{\mc M+\sqrt{2\mc M\,\Bnorm{\frac{\Sigma}{\rho\Sigma+I}}}}\log d\\
 & \lesssim\Bpar{(1+c)\mc M+\frac{1}{c}\,\Bnorm{\frac{\Sigma}{\rho\Sigma+I}}}\log d\\
 & \leq\bpar{(1+c)\,\frac{t^{2}\tr\Sigma}{N}+\frac{1}{c\rho}}\log d
\end{align*}
Therefore, with probability at least $1-\nicefrac{\fp}{d}$,
\[
\msf A_1\precsim\bpar{(1+c)\,\frac{t^{2}\tr\Sigma}{N}+\frac{1}{c\rho}}\log d\cdot(\rho\Sigma+I)\,.
\]
By solving $\frac{\log d}{c}=\veps$ and $\frac{ct^{2}\tr\Sigma}{N}=\frac{1}{c\rho}$
for $c$ and $\rho$, there exist multiples of $c$ and $\rho$ such
that 
\[
\msf A_1 \preceq\veps\Sigma+\O\bpar{\frac{t^{2}\tr\Sigma\log^{2}d}{\veps N}}\,I\,.
\]
Under the same event, we also have $-\msf A_1 \preceq\veps\Sigma+\O(\frac{t^{2}\tr\Sigma \log^{2}d}{\veps N})I$
in a similar manner.
\end{proof}
\begin{lem}[Term $\msf A_2$ and $\msf A_3$] \label{lem:covar_termBC} For $s>0$ and $t\geq2$, with probability at least $1-\nicefrac{\fp}{s}$,
\[
\Bnorm{\frac{1}{N}\sum X_{i}^{\otimes2}\ind_{B^{c}}}\leq s\,\E[X^{\otimes2}\ind_{B^{c}}]\leq3s\sqrt{\bpar{4\cpi(\pi)+\tr\Sigma}\tr\Sigma}\exp\Bpar{-\frac{t}{2}\,\bpar{\frac{\tr\Sigma}{\cpi(\pi)}}^{1/2}}\,.
\]
\end{lem}

\begin{proof}
Using Markov's inequality, we have 
\[
\P\Bpar{\Bnorm{\frac{1}{N}\sum X_{i}^{\otimes2}\ind_{B^{c}}}\geq s\,\E\Bbrack{\Bnorm{\frac{1}{N}\sum X_{i}^{\otimes2}\ind_{B^{c}}}}}\leq\frac{1}{s}\qquad\text{for any }s>0\,.
\]
The expectation can be bounded by 
\begin{align*}
\E\Bbrack{\Bnorm{\frac{1}{N}\sum X_{i}^{\otimes2}\ind_{B^{c}}}} & \leq\E\Bbrack{\frac{1}{N}\sum_{i}\norm{X_{i}}^{2}\ind_{B^{c}}}=\E[\norm X^{2}\ind_{B^{c}}]\\
 & \leq\sqrt{\E[\norm X^{4}]}\sqrt{\P\bpar{\norm X\geq t\sqrt{\tr\Sigma}}}\\
 & \leq3\sqrt{\bpar{4\cpi(\pi)+\tr\Sigma}\,\tr\Sigma}\exp\Bpar{-\frac{t}{2}\,\bpar{\frac{\tr\Sigma}{\cpi(\pi)}}^{1/2}}\qquad\text{for any }t\geq2\,,
\end{align*}
where the bounds on the first and the second term follow from Lemma~\ref{lem:4th-Poincare} and Lemma~\ref{lem:tail-Poincare} respectively. Since $\norm{\msf A_3}=\norm{\E[X^{\otimes2}\ind_{B^{c}}]} \leq \E[\norm X^{2}\ind_{B^{c}}]$, it is also bounded by the last bound above.
\end{proof}

\begin{proof}[{Proof of Lemma~\ref{lem:cov-estimation}}]
Putting these bounds together with $s=d$, with probability at least $1-\nicefrac{\fp}{d}$,
\[
\Bigl|\frac{1}{N}\sum X_{i}^{\otimes2}-\Sigma\Bigr|\preceq\veps\Sigma+\O\Bpar{\frac{t^{2}\tr\Sigma\log^{2}d}{\veps N}+\underbrace{d\,\bpar{\cpi(\pi)+\tr\Sigma}\exp\Bpar{-\frac{t}{2}\,\bpar{\frac{\tr\Sigma}{\cpi(\pi)}}^{1/2}}}_{\eqqcolon(\#)}}I\,.
\]
By setting $t\asymp2(\frac{\tr\Sigma+\cpi(\pi)}{\tr\Sigma})^{1/2}\log\frac{d\max(d,\cpi(\pi)+\tr\Sigma)}{\delta}$, we can make $(\#)\lesssim\delta$. The claim then follows by taking 
\[
N\asymp\frac{\tr\Sigma+\cpi(\pi)}{\veps\delta}\log^{2}\frac{d\max(d,\cpi(\pi)+\tr\Sigma)}{\delta}\log^{2}d\,,
\]
which completes the proof.
\end{proof}

\begin{proof}[{Proof of Theorem~\ref{thm:add-cov-Poincare}}]
Consider $\Sigma'=\frac{1}{N}\sum_{i}(X_{i}-\mu)^{\otimes2}=\frac{1}{N}\sum_{i}X_{i}^{\otimes2}-\overline{X}\mu^{\T}-\mu\overline{X}^{\T}+\mu\mu^{\T}$, where $\mu$ is the true mean $\E_\pi X$.
Now, we have 
\[
\overline{\Sigma}-\Sigma' =\frac{1}{N}\sum_{i=1}^N X_{i}^{\otimes2}-\overline{X}^{\otimes2}-\Bpar{\frac{1}{N}\sum_{i=1}^NX_{i}^{\otimes2}-\overline{X}\mu^{\T}-\mu\overline{X}^{\T}+\mu\mu^{\T}} =-(\overline{X}-\mu)^{\otimes2}\,.
\]

For the true covariance $\Sigma=\E[X^{\otimes2}]-\mu^{\otimes2}$, observe that 
\[
\overline{\Sigma}-\Sigma =\overline{\Sigma}-\Sigma'+\Sigma'-\Sigma =\underbrace{\frac{1}{N}\sum_{i=1}^{N}\bpar{(X_{i}-\mu)^{\otimes2}-\Sigma}}_{\eqqcolon\msf A}-\underbrace{(\overline{X}-\mu)^{\otimes2}}_{\eqqcolon\msf B}\,.
\]
It suffices therefore to separately control the error for the centred covariance estimator in $\msf{A}$ (Lemma~\ref{lem:cov-estimation}), and the mean estimation error in $\msf B$ (Lemma~\ref{lem:mean-estimation}).

The theorem then follows immediately upon combining Lemma~\ref{lem:mean-estimation} and~\ref{lem:cov-estimation}.
\end{proof}

\paragraph{Multiplicative version.}
The proof of this version is similar in overall with the previous one, proceeding with a truncated region $B_t :=\{x\in \Rd: \norm{x} \leq t\sqrt{\tr \Sigma}\}$ instead of $B_{\rho, t}$.

\begin{lem}\label{lem:mul-cov-estimation} Under Assumption~\ref{as:markov}, if $\lambda\geq0.99$, then with probability at least $1-\nicefrac{\fp}{d}$
we have
\[
\Bnorm{\frac{1}{N}\sum X_{i}^{\otimes2}-\Sigma}\leq \veps \norm{\Sigma}
\]
so long as $N\asymp\frac{d}{\veps^2}\,\frac{\tr\Sigma+\cpi(\pi)}{\tr\Sigma}\log^{2}\frac{d\,(\cpi(\pi)+\tr\Sigma)}{\veps \norm{\Sigma}}\log^{2}d$. 
\end{lem}

\begin{proof}
    Following the proof of Lemma~\ref{lem:covar_termA} with $\rho = 0$, we have that with probability at least $1-\nicefrac{\fp}{d}$, using Young's inequality with $c= \frac{\log d}{\veps}$,
    \[
    \norm{\msf A_1} \lesssim     
    \bpar{(1+c)\,\frac{t^{2}\tr\Sigma}{N}+\frac{1}{c}\,\norm{\Sigma}}\log d\leq
    \veps \norm{\Sigma} + \frac{t^2\tr\Sigma \log^2 d}{\veps N}\,.
    \]
    Combining this with the bounds for terms $\msf{A}_2$ and $\msf{A}_3$ in Lemma~\ref{lem:covar_termBC}, it follows that
    \[
    \Bnorm{\frac{1}{N}\sum X_{i}^{\otimes2}-\Sigma} \lesssim 
    \veps \norm{\Sigma} + \frac{t^2\tr\Sigma \log^2 d}{\veps N}
    + d\,\bpar{\cpi(\pi)+\tr\Sigma}\exp\Bpar{-\frac{t}{2}\,\bpar{\frac{\tr\Sigma}{\cpi(\pi)}}^{1/2}}\,.
    \]
    We can bound the third term by $\veps \norm{\Sigma}$ by setting $t \geq 2 \vee (\frac{\cpi(\pi)}{\tr\Sigma})^{1/2}\log\frac{d(\cpi(\pi)+\tr\Sigma)}{\veps \norm{\Sigma}}$. Under this choice $t$, the second term can be bounded by $\veps \norm{\Sigma}$ if we take 
    \[
    N\asymp \frac{t^2d\log^2 d}{\veps^2} \asymp 
    \frac{d}{\veps^2}\Bpar{1\vee \frac{\cpi(\pi)}{\tr\Sigma}\log^2\frac{d\bigl(\cpi(\pi)+\tr\Sigma\bigr)}{\veps\norm{\Sigma}}}\log^2 d\,.\qedhere
    \]
\end{proof}

\begin{proof}[Proof of Theorem~\ref{thm:mul-cov-Poincare}]
As in the proof of Theorem~\ref{thm:add-cov-Poincare}, we just combine the two bounds (i.e., errors for the mean and covariance) from Corollary~\ref{cor:mean-multi} and Lemma~\ref{lem:mul-cov-estimation}.
\end{proof}

\begin{proof}[{Proof of Corollary~\ref{cor:mul-cov-Poincare}}]
We just apply Theorem~\ref{thm:mul-cov-Poincare} after transforming
the whole system by $x\mapsto\Sigma^{-1/2}x$. Then, the covariance becomes $\Sigma=I$, and we obtain that for
$N\asymp\frac{d+\cpi(\nu)}{\veps^{2}}\log^{2}\frac{d(d+\cpi(\nu))}{\veps}\log^{2}d$,
\[
\Abs{\Sigma^{-1/2}\overline{\Sigma}\Sigma^{-1/2}-I}\preceq\veps I\,,
\]
from which the claim follows by conjugating both sides by $\Sigma^{1/2}$.
\end{proof}

\section{Details for the applications}\label{app:proofs-app}
\subsection{Isotropic rounding}\label{app:rounding}

We show that $\ino_{N_i}(\pi_i,\cdot, h_i)$ has a spectral gap at least $0.99$.

\begin{proof}[Proof of Lemma~\ref{lem:INO-setting}]
Within each while-loop, $\ino$ iterates $k_{i}N_{i}=\Otilde(cC^{4}d^{3})$ times, and the number of while-loops is at most $2\log d$, as we will show later in Lemma~\ref{lem:basic_property_isotropization}. 
Therefore, the total number of iterations of $\ino$ throughout Algorithm~\ref{alg:iterative_rounding} is $T:=\Otilde(cC^{4}d^{3})$. With the total failure probability of $\ino$ (throughout the entire algorithm) set to $\eta=\nicefrac{1}{d}$, \cite[Theorem 27]{kook2024and} requires the variance $h$ of $\ino$ to be smaller than $r_{i}^{2}(2d^{2}\log\frac{18T}{\eta})^{-1}=r_{i}^{2}(2d^{2}\log18dT)^{-1}$, justifying the choice of $h_{i}$ in Line~\ref{line:INO-setting}.

By \eqref{eq:chi-contraction} (or \cite[Theorem 23]{kook2024and}), it suffices for $\ino$ to iterate $N_{i}\lesssim h_i^{-1}\cpi(\pi_{i})$ times to ensure the spectral gap is at least $0.99$. 
Since $\cpi(\pi_{i})\lesssim \norm{\Sigma_{i}}_{\op}\log d$ (Remark~\ref{rem:PI-pushforward})  and $\norm{\Sigma_{i}}_{\op}\lesssim C^{2}d\log d$ (Lemma~\ref{lem:basic_property_isotropization}), 
$\ino_{N_{i}}(\pi_i,\cdot, h_i)$ with $N_{i}= 2^{10}C^{2}d^{3}r_{i}^{-2}\log(Ccd)$ has the desired spectral gap.
The choices of $h$ and $N$ in Line~\ref{line:last-estimation} can be justified in a similar way.
\end{proof}

\paragraph{Covariance estimation.}

Recall that Theorem~\ref{thm:add-cov-Poincare} assumes that the initial
distribution of a Markov chain is already stationary, though the actual
initial distribution $\nu=\law(X_{0})$ of Algorithm~\ref{alg:iterative_rounding}
is different from $\pi_{\K}$. We address this in Lemma~\ref{lem:cov-estimation-app1}; the conclusion is that we only pay a small additional factor in order to handle this initial bias, which can be absorbed into the other terms.

\begin{proof}[Proof of Lemma~\ref{lem:cov-estimation-app1}]
Let $\nu=\law(X_{0})$ be the law of the warm start generated by $\gc$ in Line~\ref{line:GC-warm}. As $\eu R_{\infty}(\nu\mmid\pi_{\K_{1}})\leq\log2$, then we have $\frac{\nu}{\pi_{i}}\leq2$ almost surely over $\K_{1}$.
Following a similar argument as in Lemma~\ref{lem:change-measure},
we note that all the required bounds  on bad events in
the proof of Theorem~\ref{thm:add-cov-Poincare} are multiplied by
at most $3$, and that $\E_{\nu}[\norm X^{i}]\leq3\E_{\pi_{\K_1}}[\norm X^{i}]$
for $i=2,4$. Therefore, we still have that $\Abs{\widehat{\Sigma}-\Sigma}\preceq\veps\Sigma+\delta I_{d}$ for sufficiently many samples $N$. 

As shown later in Lemma~\ref{lem:basic_property_isotropization},
$\tr\Sigma\lesssim C^{2}d^{2}$ throughout the algorithm,
and it is well-known that $\cpi(\pi_{i})\lesssim\tr\Sigma_{i}$ (due
to the log-concavity of $\pi_{i}$). Hence, with $\veps=0.1$ and $\delta=d/100$,
there exists a universal constant $c$ such that so long as $N\geq cd^{-1}\tr\Sigma\log^{6}Cd$,
\[
\Abs{\widehat{\Sigma}-\Sigma}\preceq\frac{1}{10}\,\Sigma+\frac{d}{100}\,I_{d}\,,
\]
which completes the proof.
\end{proof}

We provide quantitative control of subroutines executed within each while-loop.
\begin{proof}[Proof of Lemma~\ref{lem:basic_property_isotropization}]
The inner radius $r_{i}$ increases by a factor of at least $3/2$ each iteration. Since the algorithm starts with $r_{1}=1/4$ and ends before $r_{j}\leq\sqrt{d}$, it takes fewer than $2\log d$ iterations. 

Now we analyze how the trace and operator norm of the covariance changes each iteration.

\textbf{(1)} As per the algorithm, we have 
\begin{equation}
\Sigma_{i+1}=M_{i}\Sigma_{i}M_{i}\label{eq:Ak_update}
\end{equation}
Since $P_{i}$ is an orthogonal projection matrix,
\[
\tr\Sigma_{i+1} =\tr(\Sigma_{i}^{1/2}M_{i}^{2}\Sigma_{i}^{1/2})=\tr\bpar{\Sigma_{i}^{1/2}(I+3P_{i})\Sigma_{i}^{1/2}}\le4\tr\Sigma_{i}\,.
\]
As $r_{i+1}=2(1-\log^{-1}d)\,r_{i}$, we have that $\nicefrac{\tr\Sigma_{i}}{r_{i}^{2}}$ increases by a factor of $(1-\nicefrac{1}{\log d})^{-2}$ per iteration. Thus, it can increase by up to $10$ times since there are $2\log d$
many iterations. Therefore, we have $\tr\Sigma_{i}\le10r_{i}^{2}\tr\Sigma_{1}=10r_{i}^{2}C^{2}d$
for all $i$. 

\textbf{(2)} Note that $\|\Sigma_{1}\|_{\op}\leq\tr\Sigma_{1}\leq C^{2}d$.
Recall that $A^{\T}A$ and $AA^{\T}$ share the same spectrum (i.e.,
the same set of non-zero eigenvalues). Hence, it follows from \eqref{eq:Ak_update}
that 
\begin{align}
\|\Sigma_{i+1}\|_{\op} & =\|M_{i}\Sigma_{i}M_{i}\|_{\op}=\|\Sigma_{i}^{1/2}M_{i}^{2}\Sigma_{i}^{1/2}\|_{\op}=\|\Sigma_{i}^{1/2}(I+3P_{i})\Sigma_{i}^{1/2}\|_{\op}\nonumber\\
 & \leq\|\Sigma_{i}\|_{\op}+3\|\Sigma_{i}^{1/2}P_{i}\Sigma_{i}^{1/2}\|_{\op}=\|\Sigma_{i}\|_{\op}+3\|P_{i}\Sigma_{i}P_{i}\|_{\op}\label{eq:op-control}\,.
\end{align}
Using $k_{i}=10cr_{i}^{2}C^{2}\log^{6}C^{2}d$ many samples (which is at least $cd^{-1}\tr\Sigma_{i}\log^{6}Cd$ due to \textbf{(2)}), Lemma~\ref{lem:cov-estimation-app1} ensures that $0.9\Sigma_{i}-\frac{d}{100}\,I_{d}\preceq\widehat{\Sigma}_{i}\preceq1.1\Sigma_{i}+\frac{d}{100}\,I_{d}$
with probability at least $1-\nicefrac{\fp}{d}$. 
Conjugating by the projection matrix $P_{i}$, we have using a somewhat lazy bound that
\[
P_{i}\Sigma_{i}P_{i}\preceq\frac{3}{2}\,P_{i}\widehat{\Sigma}_{i}P_{i}+\frac{d}{50}\,P_{i}\,.
\]
Let $\widehat{\Sigma}_{i}=U^{\T}DU$ be its spectral decomposition
with orthogonal matrix $U$. Then, under this decomposition,
both $UP_{i}U^{\T}$ and $D=U\widehat{\Sigma}_{i}U^{\T}$ are diagonal,
where the entry in $UP_{i}U^{\T}$'s diagonal is set to $0$ if the corresponding
diagonal entry of $U\widehat{\Sigma}_{i}U^{\T}$ is larger than $d$, and $1$ otherwise. 
Letting $\min(dI_d,D)$ be the entry-wise minimum of $dI_d$ and $D$, we have that
\[
UP_{i}\Sigma_{i}P_{i}U^{\T}\preceq\frac{3}{2}\,\min(dI_d,D)+\frac{d}{50}I_{d}\preceq2dI_{d}\,,
\]
and $\norm{P_{i}\Sigma_{i}P_{i}}_{\op}\leq 2d$. 
Substituting this back to \eqref{eq:op-control}, we obtain $\norm{\Sigma_{i+1}}_{\op} \leq\|\Sigma_{i}\|_{\op} + 6d$.
Therefore, the claim follows from recursing $i$ times.
\end{proof}

\paragraph{Doubling of the inner radius.}
We recall an additional technical lemma.
\begin{lem}[{\cite{lovasz2006simulated}, Lemma 3.4}] \label{lem:iso-inner-outer-ball}Let
$\mc K$ be convex in $\Rd$ with centroid $\mu$ and covariance
$\Sigma$ satisfying $r^{2}I_{d}\preceq\Sigma\preceq R^{2}I_{d}$.
Then, 
\[
B_{r}(\mu)\subset\mc K\subset B_{Rd}(\mu)\,.
\]
\end{lem}

Lastly, we show that under the update rule of $r_{i+1} \gets 2(1-\nicefrac{1}{\log d})r_i$, the inner radius actually almost doubles if all the covariance estimations thus far have been accurate.
\begin{proof}[Proof of Lemma~\ref{lem:inner-rad-double}]
We show that when $B_{r_{i}}(c_{i})\subseteq\K_{i}$ for some $c_{i}\in\mc K_{i}$, there exists some center $c_{i+1}\in\mc K_{i+1}$ such that $\mc K_{i+1}$
contains $B_{r_{i+1}}(c_{i+1})$ with $r_{i+1}=2(1-\nicefrac{1}{\log d})\,r_{i}$.
By Lemma~\ref{lem:iso-inner-outer-ball}, $\mc K_{i}$ also contains
the ellipsoid $\{x:\norm{x-\mu_{i}}_{\Sigma_{i}^{-1}}^{2}\leq1\}$
for the mean $\mu_{i}$ and covariance $\Sigma_{i}$ of $\mc K_{i}$
with respect to $\pi_{i}$. Recall that
\[
M_{i}=I+P_{i}\,,\quad T_{i+1}=M_{i}T_{i}\,,\quad\Sigma_{i+1}=M_{i}\Sigma_{i}M_{i}\,.
\]
We first focus on the case where the two centers $c_{i}$ and $\mu_{i}$ are different.

Let $\widehat{\Sigma}_{i}=U^{\T}D_{i}U$ be the spectral decomposition of the estimated covariance $\widehat{\Sigma}_{i}$, where $U\in\Rdd$ is an orthogonal matrix, and $D_{i}\in\Rdd$ is a diagonal matrix with eigenvalues on the diagonal
in decreasing order. Under the transformation $x\mapsto y:=M_{i}x$
(i.e., $\mc K_{i}\to\mc K_{i+1}$), the new convex body $\mc K_{i+1}$ contains two ellipsoids: defining $c_{i}':=M_{i}c_{i}$ and $\mu_{i}':=M_{i}\mu_{i}(=\mu_{i+1})$,
\[
\mc A:\{y \in \R^d: (y-c_{i}')^{\T}(I+P_{i})^{-2}(y-c_{i}')\leq r_{i}^{2}\}\,,\qquad\text{and}\qquad\mc B:\{y \in \R^d: (y-\mu_{i}')^{\T}\Sigma_{i+1}^{-1}(y-\mu_{i}')\leq1\}\,.
\]

We now work with a new coordinate system in $z:=Uy$ for the
ease of analysis. Under this new system, there exists $0\leq r\leq d$
such that the $r$ largest eigenvalues of $U\widehat{\Sigma}_{i}U^{\T}=D_{i}$
(corresponding to bases $\{e_{1},\dots,e_{r}\}$) are larger than the
threshold $d$ and that the remaining $d-r$ eigenvalues are smaller than
$d$. Note that under the $z$-coordinate system, $P_{i}$ is given
by 
\[
UP_{i}U^{\T}=\left[\begin{array}{cc}
0_{r\times r}\\
 & 1_{(d-r)\times(d-r)}
\end{array}\right]\ \text{since }D_{i}=\left[\begin{array}{cc}
[\geq d]_{r\times r}\\
 & [<d]_{(d-r)\times(d-r)}
\end{array}\right]\,.
\]
Hence, under the $z$-coordinate system, the two ellipsoids above can be written as
\begin{align}
(I+P_{i})^{2} & \to(I+UP_{i}U^{\T})^{2}=\left[\begin{array}{cc}
1_{r\times r}\\
 & 2_{(d-r)\times(d-r)}
\end{array}\right]^{2}\,,\label{eq:A-ellipsoid}\\
\Sigma_{i+1} & \to U\Sigma_{i+1}U^{\T}=U(I+P_{i})U^{\T}\cdot U\Sigma_{i}U^{\T}\cdot U(I+P_{i})U^{\T}\nonumber \\
 & \quad =\underbrace{\left[\begin{array}{cc}
1_{r\times r}\\
 & 2_{(d-r)\times(d-r)}
\end{array}\right]}_{\eqqcolon\overline{D}}U\Sigma_{i}U^{\T}\left[\begin{array}{cc}
1_{r\times r}\\
 & 2_{(d-r)\times(d-r)}
\end{array}\right]\,.\label{eq:B-ellipsoid}
\end{align}

Let $c_{\mc A}$ and $c_{\mc B}$ denote the centers of $\mc A$-ellipsoid and $\mc B$-ellipsoid in the $z$-coordinate system. 
We define $\mc S=\spanning(\{e_{r+1},\dots,e_{d}\})$ (i.e. the span of the deficient axes) and $\mc S^{\perp}=\spanning(\{e_{1},\dots,e_{r}\})$ (i.e. the span of the sufficient axes). 

\textbf{Case I: $0<r<d$.}
We first show that in the $z$-coordinate system, $\mc K_{i+1}$ contains two lower-dimensional balls
\[
\msf{B}_{\mc S}:= B_{2lr_i}(c_l)\cap \mc S\quad \& \quad \msf{B}_{\mc S^\perp}:= B_{\frac{9(1-l)}{10}\,d^{1/2}}(c_l) \cap \mc S^\perp\,, 
\]
where $c_{l}\coloneqq lc_{\mc A}+(1-l)c_{\mc B}$ and $l \coloneqq (1-\nicefrac{1}{\log d})^{1/2}$.
 
By~\eqref{eq:A-ellipsoid}, the part of the $\mc A$-ellipsoid along $\mc S$ (i.e., $\mc A\cap(\mc S+c_{\mc A})$) contains $B_{2r_{i}}(c_{\mc A})$.
By convexity of $\mc K_{i+1}$, it contains $B_{2lr_{i}}(c_{l})$ along the hyperplane $\mc S+c_{\ell}$, since $\mc K_{i+1}$ contains $c_{\mc B}$.
As for $\mc S^\perp$, by Lemma~\ref{lem:cov-estimation-app1}, with probability at least $1-d^{-1}$, $\widehat{\Sigma}_{i}\preceq\tfrac{11}{10}\Sigma_{i}+\frac{d}{100}\,I_{d}$.
Conjugating by $U^{\T}\overline{D}$, this is equivalent to 
\[
\overline{D}D_{i}\overline{D}\preceq\frac{11}{10}\,\overline{D}U\Sigma_{i}U^{\T}\overline{D}+\frac{d}{100}\,\overline{D}^{2}
\underset{\eqref{eq:B-ellipsoid}}{=}\frac{11}{10}\,U\Sigma_{i+1}U^{\T}+\frac{d}{100}\,\overline{D}^{2}\,.
\]
Then the projection onto $\mc S^{\perp}$
(i.e., taking the top-left $r\times r$ block matrix) results in
\[
(U\Sigma_{i+1}U^{\T})\vert_{\mc S^{\perp}}\succeq\frac{9}{10}\,dI_{r}\,,
\]
so the $\mc B$-ellipsoid along $\mc S^{\perp}$ contains $B_{0.9 d^{1/2}}(c_{\mc B})$.
Since the $\mc A$-ellipsoid along $\mc S^{\perp}$ contains $B_{r_{i}}(c_{\mc A})$
as seen in~\eqref{eq:A-ellipsoid}, the convexity of $\mc K_{i+1}$
implies that it contains a ball centered at $c_{l}$ of radius $lr_{i}+(1-l)\cdot \tfrac{9}{10} d^{1/2} \ge \tfrac{9(1-l)}{10}d^{1/2}$.

We now prove that $\K_{i+1}$ contains a ball of radius roughly $2r_i$ centered at $c_l$, by showing that any point in such a ball can be written as a convex combination of $\msf{B}_{\mc S}$ and $\msf{B}_{\mc S^\perp}$.
More specifically, we may assume $c_{l}=0$ by translation and denote by $P$ the projection to $\mc S=\mc S + c_l$. For any $x$ with $\norm x\leq2clr_{i}$ and $c$ to be determined, we write 
\[
x=(1-t)\,\frac{(I-P)x}{1-t}+t\,\frac{Px}{t}\quad\text{for }t\in(0,1)\,.
\]
For $u:=(I-P)x$ and $v:=Px$, we will show that if $\norm u^{2}+\norm v^{2}=\norm x^{2}\leq4c^{2}l^{2}r_{i}^{2}$, then 
\begin{align*}
    &\exists t\in(0,1):\quad \frac{\norm u}{1-t}\leq\frac{9(1-l)}{10}\,d^{1/2}\qquad\text{and}\qquad\frac{\norm v}{t}\leq2lr_{i}\\
    \Longleftrightarrow\ &\exists t\in(0,1): \frac{\norm v}{2lr_{i}}\leq t\leq1-\frac{10\norm u}{9(1-l)d^{1/2}}\\
    \Longleftrightarrow\ &\frac{10\norm u}{9(1-l)d^{1/2}}+\frac{\norm v}{2lr_{i}}=\frac{20lr_{i}}{9(1-l)d^{1/2}}\,a+b \leq 1 \text{ for } a:=\frac{\norm u}{2lr_{i}}\text{ and } b:=\frac{\norm v}{2lr_{i}}\,.
\end{align*}
To this end, it suffices to show that the maximum of the following problem is at most $1$:
\[
\max\frac{20lr_{i}}{9(1-l)d^{1/2}}\,a+b\quad\text{subject to }a^{2}+b^{2}=c^{2}\,.
\]
Using the Lagrange multiplier method, the maximum is attained if $(a,b)=\lda(\tfrac{20lr_{i}}{9(1-l)d^{1/2}},1)$ for some $\lda>0$. Solving $a^2 + b^2 = c^2$ with this condition, we have $\lda=c(1+\tfrac{400l^{2}r_{i}^{2}}{81(1-l)^{2}d})^{-1/2}$. Thus, the maximum is
\[
\frac{20lr_{i}}{9(1-l)d^{1/2}}\,a+b=c\,\bpar{1+\frac{400l^{2}r_{i}^{2}}{81(1-l)^{2}d}}^{1/2}\,,
\]
which can be bounded by $1$ by setting $c=(1+\frac{400l^{2}r_{i}^{2}}{81(1-l)^{2}d})^{-1/2}$.

We now show that $c^{-2} \leq l^{-2}$ (so $c\geq l$). Using $2^{10}r_{i}^{2}\log^{4}d\leq d$ and $l\leq 1-\nicefrac{1}{2\log d}\leq 1$, we have 
\[
c^{-2} = 1+\frac{400l^{2}r_{i}^{2}}{81(1-l)^{2}d}
\leq 1+\frac{400l^{2}}{81\cdot2^{10}(1-l)^{2}\log^{4}d}
\leq 1+\frac{1}{12\log^{2}d}\,.
\]
As $l^{-2} =1 + \frac{1}{\log d -1}$, we clearly have $c^{-2}\leq l^{-2}$. Therefore,
\[
\msf{inrad}(\mc K_{i+1}) \geq 2clr_i \geq 2l^2r_i.
\]
\textbf{Case II: $r=0$ or $d$.}
When $r=0$, the $\mc A$-ellipsoid is simply the $d$-dimensional
$B_{2r_{i}}(c_{\mc A})$. If $r=d$, then it means that the $\mc B$-ellipsoid
contains the $d$-dimensional ball $B_{\frac{9}{10}d^{1/2}}(c_{\mc B})$.
Since $2^{10}r_{i}^{2}\log^{4}d\leq d$, we have $\frac{9}{10}d^{1/2}\geq2l^{2}r_{i}$, where $l$ is as before. Combining these two cases justifies $\msf{inrad}(\K_{i+1})\geq2l^{2}r_{i}$.

Lastly, if the two centers $c_i$ and $\mu_i$ are the same, then $c_{\mc A} = c_{\mc B}$. In such case, the same argument goes through (with an even larger radius).
\end{proof}

\paragraph{Final guarantee.}
Putting these together, we can prove our main result in this section.

\begin{proof}[{Proof of Lemma~\ref{lem:rounding-final}}]
We would like to exclude three bad events --- (1) failure of $\gc$ in Line~\ref{line:GC-warm}, (2) failure of $\ino$ throughout the algorithm, and (3) failure of covariance estimation in Line~\ref{line:mean-cov-estimate} within each while-loop and in Line~\ref{line:last-estimation}.

As for (1), we can simply set the failure of the warm-start generation to $\nicefrac{\fp}{d}$ with logarithmic overhead in $d$ \cite{kook2024renyi}. As for (2), we already picked $h_i$ and $N_i$ in Line~\ref{line:INO-setting} so that the total failure probability of $\ino$ is at most $\nicefrac{\fp}{d}$ (see the proof of Lemma~\ref{lem:INO-setting}).
As for (3), by Lemma~\ref{lem:basic_property_isotropization}, the failure probability throughout the while-loop is at most $\nicefrac{\fp\log d}{d}$ by the union bound. As for the final covariance estimation in Line~\ref{line:last-estimation}, by the multiplicative
form of covariance estimation guarantees (Corollary~\ref{cor:mul-cov-Poincare}),
there exists a universal constant $c'>0$ such that $\Abs{\widehat{\Sigma}-\Sigma}\preceq\frac{1}{10}\,\Sigma$ with probability at least $1-\frac{\fp}{d}$ when $N\geq c'd\log^{6}d$.

Putting all things together, the algorithm succeeds with probability
at least $1-\nicefrac{\fp \log d}{d}-\nicefrac{\fp}{d} \geq1-\nicefrac{\fp}{\sqrt{d}}$, ensuring that the covariance of the transformed convex body satisfies 
\[
\frac{9}{10}\,I_{d}\preceq\widehat{\Sigma}^{-1/2}\Sigma\widehat{\Sigma}^{-1/2}\preceq\frac{11}{10}\,I_{d}\,.
\]
Therefore, $\widehat{\Sigma}^{-1/2}(\mc K - \widehat{\mu})$ is $2$-isotropic.

As for the query complexity, Algorithm~\ref{alg:iterative_rounding}
starts with $\gc$ for the warm-start generation, using $\Otilde(C^{2}d^{3})$ membership queries. In each while-loop,
since $\ino$ from an $\O(1)$-warm start iterates $k_{i}N_{i}=\Otilde(cC^{4}d^{3})=\Otilde(C^{4}d^{3})$
times, it uses $\Otilde(C^{4}d^{3})$ membership queries. Since the
number of while-loops is at most $2\log d$, the total query complexity
is $\Otilde(C^{4}d^{3})$.
\end{proof}

\subsection{Covariance estimation for unconstrained distributions}\label{app:unconstrained}

The following ensures that we can tractably initialize our algorithm not too far from the target.
\begin{lem}[{Initialization,~\cite[Lemma 32]{chewi2021analysis}}]\label{lem:initialization}
    Suppose $\nabla V(0) = 0$ and $V$ is $\beta$-smooth (Assumption~\ref{as:smoothness}). Let  $\pi \propto \exp(-V)$ and $\mf m = \E_\pi \norm{\cdot}$. Then, initialization $\mu_0 = \mc N(0,\tfrac{1}{2\beta}I_d)$ satisfies
    \[
        \log \chi^2(\mu_0 \mmid \pi) \lesssim 1 + \beta + V(0) - \min V + d \log (\mf m^2 \beta)\,.
    \]
\end{lem}
The assumption that $\nabla V(0) = 0$ is standard, as a local extremum can be found by optimization algorithms, which have complexity dominated by that of the sampling procedure. Assuming mild values for all parameters of interest, this gives a bound of $\Otilde(d \vee \beta)$.

\begin{proof}[{Proof of Lemma~\ref{lem:alg-iid}}]
    Here, we elucidate a few details when adapting the results of~\cite{altschuler2024faster}. The following steps are already present in their analysis of the proximal sampler, but we reemphasize the derivation here for full clarity. Namely, by the analysis of~\cite[Appendix D.4]{altschuler2024faster}, the condition number of the backwards part of the proximal sampler is $\Theta(1)$. Furthermore, since the forward part of the proximal sampler is implemented exactly, then denoting $P_{\mrm{forward}}, P_{\mrm{backward}}$ for the two exact kernels with step size $h \asymp 1/\beta$, and $\hat P_{\mrm{backward}}$ for the approximate kernel from~\cite{altschuler2024faster},
    \[
        \KL(\delta_x \hat P \mmid \delta_x P) \leq \KL(\delta_x P_{\mrm{forward}} \mmid \delta_x P_{\mrm{forward}}) + \E_{y \sim \delta_x P_{\mrm{forward}}} \KL(\delta_y \hat P_{\mrm{backward}} \mmid \delta_y P_{\mrm{backward}}) \leq \varepsilon\,,
    \]
    using~\cite[Theorem D.1]{altschuler2024faster} as claimed. 
\end{proof}

\begin{proof}[Proof of Lemma~\ref{lem:tv-unconstrained-chain}]
    For any $z_1 \in \R^d$ and $m \in \mathbb N$, using the data-processing inequality and then the chain rule for $\KL$,
    \begin{align*}
        \KL(\delta_{z_1} \hat P^m \mmid \delta_{z_1} P^m) 
        &\leq \KL(\delta_{z_1} \hat P \mmid \delta_{z_1} P) + \E_{z_2 \sim \delta_{z_1} \hat P} \KL(\delta_{z_2} \hat P \mmid \delta_{z_2} \hat P) \\
        &\qquad + \cdots + \E_{z_m \sim \delta_{z_1} \hat P^{m-1}} \KL(\delta_{z_m} \hat P \mmid \delta_{z_m} \hat P)\,.
    \end{align*}
    This heavily uses the fact that the joint distribution arises from a Markov chain.
    
    The query complexity to bound the left side by $\varsigma \in (0, \nicefrac{1}{2}]$ is $\Otilde(m d^{1/2} \log^3 \frac{1}{\varsigma})$, obtained by choosing an error of $\varsigma/m$ for the kernel and then applying Lemma~\ref{lem:alg-iid} $m$-times. Then,
    \begin{align*}
        2\,\norm{\nu_{1:K} - \hat \nu_{1:K}}_{\tv}^2 &\leq \KL(\hat \nu_{1:K} \mmid \nu_{1:K}) \\
        &\leq \KL(\hat \nu \mmid \nu) + \E_{x_1 \sim \hat \nu} \KL(\delta_{x_1} \hat P^n \mmid \delta_{x_1} P^n) \\
        &\qquad + \cdots + \E_{x_{K-1} \sim \hat \nu \hat P^{(K-1)n}} \KL(\delta_{x_{K-1}} \hat P^{n} \mmid \delta_{x_{K-1}} P^n)\,. 
    \end{align*}
    Therefore, the query complexity of bounding the left side by $2\delta^2$ is  $\Otilde((Kn + n_0) d^{1/2} \log^3 \frac{Kn+n_0}{\delta})$.
\end{proof}

\begin{proof}[Proof of Theorem~\ref{thm:unconstrained-covar} and Remark~\ref{rmk:unconstrained-iid}]

    For the choices of $h \asymp 1/\beta$, $n = \O(\kappa)$ given, the spectral gap of $P^n$ will be bounded below by $0.99$. Thus, choosing $K = \widetilde \Theta(\frac{d \phi}{\varepsilon^2})$, the covariance estimator using samples from the exact kernel $P^n$ will concentrate around $\Sigma$. Precisely, if
    \[
        \overline \Sigma = \frac{1}{K} \sum_{j=1}^K \Bigl(Y_j - \frac{1}{K} \sum_{k=1}^K Y_k\Bigr)^{\otimes 2} \eqqcolon F(Y_1, \ldots, Y_K)\,, 
    \]
    where $Y_1 \sim \nu$ and $Y_{i+1} \sim \delta_{Y_i} P^n$, then Theorem~\ref{thm:mul-cov-Poincare} combined with Lemma~\ref{lem:change-measure} tell us for our choice of parameters that with probability $1-\nicefrac{2}{3d}$,
    \[
        \norm{\overline \Sigma - \Sigma}_{\op} \leq \veps\norm{\Sigma}_{\op}\,.
    \]
    Here, we use that our initial distribution $\nu = \mu_0 P^{n_0}$ comes from applying $P$ enough iterations that $\chi^2(\nu \mmid \pi) \lesssim 1$, and so via Lemma~\ref{lem:initialization} our choice of $n_0$ is sufficient. Note that we can also write $\hat \Sigma = F(X_1, \ldots, X_k)$, where $\{X_j\}_{j \leq K}$ is generated using $\hat P^n$ instead.

    Denote the undesirable event on $\R^{K \times d}$ by
    \[
        \mc A \deq \bigl\{(X_1, \ldots, X_K) \in \R^{d\times K}: \norm{F(X_1, \ldots, X_K)- \Sigma}_{\op} > \varepsilon \norm{\Sigma}_{\op} \bigr\}\,.
    \]
    Under the approximate kernel $\hat P^n$, we have via Lemma~\ref{lem:tv-unconstrained-chain} and the definition of the $\tv$ distance that
    \[
        \hat \nu_{1:K}(\mc A) \leq \nu_{1:K}(\mc A) + \abs{\nu_{1:K}(\mc A) - \hat \nu_{1:K}(\mc A)}
        \leq \nu_{1:K}(\mc A) + \norm{\nu_{1:K} - \hat \nu_{1:K}}_{\tv}\,.
    \]
    Choosing $\delta \asymp \nicefrac{\fp}{d}$ in Lemma~\ref{lem:tv-unconstrained-chain} to make this $\tv$ distance bounded by $\nicefrac{\fp}{d}$, and adding it to our bound under $\nu_{1:K}$ from earlier, we obtain a bound of $1-\nicefrac{\fp}{d}$ on the probability of failing to estimate the covariance under the approximate implementation of the kernel.
    
    The total expected query complexity is then given by
    \[
        N \leq \Otilde\bigl(K \kappa d^{1/2} \log^3(K+n_0)\bigr) + \Otilde\bigl(n_0 d^{1/2} \log^3(K + n_0)\bigr)\,,
    \]
    where the second term is the cost of generating the initial iterate, and the first term comes from the generation of the remaining iterates. Putting all these together concludes the proof.

    For the bound given in Remark~\ref{rmk:unconstrained-iid}, we note that we will again need $K = \widetilde \Theta(\frac{d \phi}{\varepsilon^2})$ iterates in total. By the same principle as Lemma~\ref{lem:tv-unconstrained-chain}, it suffices that each iterate is $\mc O(\frac{1}{Kd})$ close to the target in $\sqrt{\KL}$, and we should again choose $\varrho$ identically to the Markovian case. To produce a single iterate, we require $\Otilde(n_0 d^{1/2} \polylog (K+ n_0))$ queries in expectation. Putting these all together concludes the proof.
\end{proof}

\section{Concluding remarks}

In summary, we have established in this work that Markov chains can estimate the covariance of a target distribution under the joint assumptions of~\eqref{eq:poincare} and a spectral gap. We demonstrated two applications where this has a substantial benefit in terms of query complexity bounds. In particular, we demonstrated its relevance to an iterative rounding algorithm, which is a cornerstone of the literature for convex body sampling.

We foresee several primary extensions to our efforts that may be worthwhile. Firstly, in the proof of Lemma~\ref{lem:covar_termBC}, we only obtained a polynomial tail bound via Markov's inequality inequality due to dependence of samples, whereas~\cite{adamczak2010quantitative} was able to find a subexponential bound. Closing this gap would render our result completely analogous with the prior literature in the i.i.d. setting.

An interesting complementary approach may be to establish a Poincar\'e inequality or other functional inequality directly on the joint distribution of the iterates $\{X_j\}_{j \leq N}$. Note that in the i.i.d. case, a Poincar\'e inequality follows immediately by tensorization.

Finally, one may extend this to higher-order moment tensors, or even to exponential quantities. To this end, it may be interesting to consider what additional guarantees may follow as a result of a log-Sobolev inequality (LSI). LSIs are the other primary functional inequality considered in sampling, and generally lead to stronger concentration results.

\bibliographystyle{alpha}
\bibliography{main}

\appendix

\section{Finer bound on the tail}\label{app:finer-tail}

\begin{lem}
Under the same assumptions as Lemma~\ref{lem:covar_termA}, we in fact have the improved bound  
\[
\P\Bpar{\frac{1}{N}\sum_{i}\norm{X_{i}}^{2}\ind_{B^{c}}-\E\Bbrack{\frac{1}{N}\sum_{i}\norm{X_{i}}^{2}\ind_{B^{c}}}\geq s\rho}\leq\frac{1}{s^{2}}\,.\qedhere
\]
Here, $\rho$ can be bounded as $
\rho^{2}\leq\frac{1}{N}(1+\frac{2}{\lambda})(\tr\Sigma+C_{\PI}(\pi))^{2}e^{-t/2}.$
\end{lem}

\begin{proof}
Note that 
\[
\Bnorm{\frac{1}{N}\sum X_{i}^{\otimes2}\ind_{B^{c}}}\leq\frac{1}{N}\sum\sup_{v}|X_{i}\cdot v|^{2}\ind_{B^{c}}=\frac{1}{N}\sum_{i}\norm{X_{i}}^{2}\ind_{B^{c}}\,.
\]
Consider the centered function 
$
f(X_{i})\deq\{\norm{X_{i}}^{2}\ind_{B^{c}}-\E\norm{X_{i}}^{2}\ind_{B^{c}}\}$,
in which case we can write 
\begin{align*}
\E\Bigl[\Big\{ & \frac{1}{N}\sum_{i=1}^{N}\norm{X_{i}}^{2}\ind_{B^{c}}-\E\Bbrack{\frac{1}{N}\sum_{i=1}^{N}\norm{X_{i}}^{2}\ind_{B^{c}}}\Bigr\}^{2}\Bigr]\\
 & =\frac{1}{N^{2}}\E[(\sum_{i=1}^{N}f(X_{i}))^{2}] =\frac{1}{N^{2}}\sum_{i=1}^{N}\E[(f(X_{i}))^{2}]+\frac{2}{N^{2}}\sum_{i=1}^{N}\E[\sum_{\substack{j>i\\
j\in[N]
}
}f(X_{i})f(X_{j})]\\
 & =\frac{1}{N}\E[(f(X_{i}))^{2}]+\frac{2}{N^{2}}\sum_{i=1}^{N}\E[f(X_{i})\sum_{\substack{j>i\\
j\in[N]
}
}\E[f(X_{j})|X_{i}]]\\
 & =\frac{1}{N}\E[(f(X_{i}))^{2}]+\frac{2}{N^{2}}\sum_{i=1}^{N}\E[f(X_{i})\sum_{\substack{k>0\\
k\in[N-i]
}
}P^{k}f(X_{i})]\,.
\end{align*}
Consider just the second term. If we observe just the even summands, we see that they can be bounded as
\[
\E\Bigl[f(X_{i})\sum_{\substack{k>0\\
k\in[N-i]
}
}P^{2k}f(X_{i})\Bigr]=\sum_{\substack{k>0\\
k\in[\lfloor{N-i}\rfloor]
}
}\E[\abs{P^{k}f}^{2}]\leq\sum_{k>0}\lambda^{k}\E[\abs{f}^{2}]=\frac{1}{1-\lambda}\,\E[\abs{f}^{2}]\,,
\]
The odd terms can be handled via Cauchy--Schwarz, $\E[fP^{k+1}f]\leq\sqrt{\E f^{2}\E[(P^{2k+2}f)^{2}]}\leq\lambda^{k+1}\E[f^{2}]$.
Thus, the effect is to approximately double the bound above.

Note that $\E[(f(X_{i}))^{2}]$ can be bounded by 
\begin{align*}
\E[\norm{X_{i}}^{4}\ind_{B^{c}}] & =\E[\norm X^{4}\ind_{B^{c}}]\leq\sqrt{\E[\norm X^{8}]}\sqrt{\P\bpar{\norm X\geq t\sqrt{\tr\Sigma}}}\\
 & \lesssim\bigl(\tr\Sigma+C_{\PI}(\pi)\bigr)^{2}\cdot e^{-t/2}\,,
\end{align*}
since 
\[
\Pr(\norm X-\E\norm X>t)\lesssim\exp\bigl(-\frac{t}{\sqrt{C_{\PI}(\pi)}}\bigr)\,,
\]
and integrating, bounding $\E\norm X\lesssim\sqrt{\tr\Sigma}$. The
mean can be bounded by 
\begin{align*}
\E\Bbrack{\frac{1}{N}\sum_{i}\norm{X_{i}}^{2}\ind_{B^{c}}} & =\E[\norm X^{2}\ind_{B^{c}}]\leq\sqrt{\E[\norm X^{4}]}\sqrt{\P\bpar{\norm X\geq t\sqrt{\tr\Sigma}}}\lesssim\tr\Sigma\cdot e^{-t/2}\,.
\end{align*}
This implies that 
\[
\rho^{2}:=\var\Bpar{\frac{1}{N}\sum_{i=1}^{N}f}\leq\frac{1}{N}\bigl(1+\frac{2}{\lambda}\bigr)\bigl(\tr\Sigma+C_{\PI}(\pi)\bigr)^{2}e^{-t/2}\,.
\]
Recall that $\lambda$ is an absolute constant which is bounded
away from zero. 

Finally, by Chebyshev's inequality, 
\[
\P\Bpar{\frac{1}{N}\sum_{i}\norm{X_{i}}^{2}\ind_{B^{c}}-\E\Bbrack{\frac{1}{N}\sum_{i}\norm{X_{i}}^{2}\ind_{B^{c}}}\geq s\rho}\leq\frac{1}{s^{2}}\,.\qedhere
\]
\end{proof}

\end{document}